\documentclass[leqno]{amsart}
\usepackage[utf8]{inputenc}

\usepackage{amsmath}
\usepackage{amsthm}

\usepackage{enumerate}

\usepackage{mathtools}

\usepackage{relsize}

\usepackage[colorlinks=false]{hyperref}

\usepackage[charter]{mathdesign}
\usepackage{xy}
\xyoption{all}

\makeatletter
\g@addto@macro\bfseries{\boldmath}
\makeatother

 \numberwithin{equation}{section}

\theoremstyle{plain}
    \newtheorem{theorem}[equation]{Theorem}
    
    \newtheorem{corollary}[equation]{Corollary}

    \newtheorem*{theorem*}{Theorem}
    \newtheorem*{proposition*}{Proposition}
    \newtheorem*{corollary*}{Corollary}
    \newtheorem*{lemma*}{Lemma}
    \newtheorem*{conjecture*}{Conjecture}
    \newtheorem{definition-theorem}[equation]{Definition/Theorem}
    \newtheorem{definition-lemma}[equation]{Definition/Lemma}
    
    \newtheorem{question}[equation]{Question}
\theoremstyle{definition}
    \newtheorem{definition}[equation]{Definition}
    \newtheorem{example}[equation]{Example}
    \newtheorem{examples}[equation]{Examples}

    \newtheorem{remark}[equation]{Remark}

    \newcommand{\N}{\mathbb{N}}
    \newcommand{\Z}{\mathbb{Z}}

\renewcommand{\phi}{\varphi}
\let\epsilon\varepsilon

\newcommand{\into}{\hookrightarrow}

\newcommand{\id}{\mathrm{id}}

    \DeclareMathOperator{\End}{End}

\newcommand{\source}{\operatorname{s}}
\newcommand{\target}{\operatorname{t}}

\newcommand{\category}{\mathsf}
\newcommand{\functor}{\mathcal}

\newcommand{\Mod}{\category{Mod}}
\newcommand{\CKL}{\category{Diag}^{\sqcup}}
\newcommand{\opp}{\mathrm{opp}}

\newcommand{\Diag}{\category{Diag}}
\newcommand{\PS}{\operatorname{PS}}
\newcommand{\BF}{\operatorname{BF}}

\newcommand{\open}{\circ}
\newcommand{\closed}{\bullet}

\title{Flow equivalence of diagram categories and Leavitt path algebras}
\author{Tyrone Crisp}
\address{Department of Mathematics \& Statistics, University of Maine.
5752 Neville Hall, Room 333.
Orono, ME 04469 USA}
\email{tyrone.crisp@maine.edu}

\author{Davis MacDonald}
\email{davis.macdonald@maine.edu}

 \date{June 2022}

\begin{document}

\begin{abstract}
Several constructions on directed graphs originating in the study of flow equivalence in symbolic dynamics (e.g., splittings and delays)  are known to preserve the Morita equivalence class of Leavitt path algebras over any coefficient field $\mathbb{F}$. We prove that many of these equivalence results are not only independent of $\mathbb{F}$, but are largely independent of linear algebra altogether. We do this by formulating and proving generalisations of these equivalence theorems in which the category of $\mathbb{F}$-vector spaces is replaced by an arbitrary category with binary coproducts, showing that the Morita equivalence results for Leavitt path algebras depend only on the ability to form direct sums of vector spaces. We suggest that the framework developed in this paper may be useful in studying other problems related to Morita equivalence of Leavitt path algebras.
\end{abstract}

\maketitle

\section{Introduction}

Suppose that the picture below represents a communication network, with each vertex  in the picture representing a node in the network, and each directed edge representing a one-way channel along which information flows from the source node to the target node. 
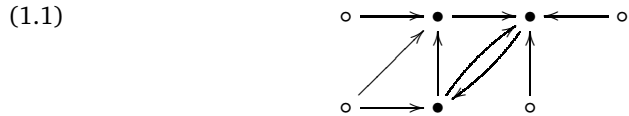
\begin{equation}\label{eq:intro-graph}
\xymatrix{
\open \ar[r] & \closed \ar[r] & \closed \ar@/^3pt/[dl] & \open \ar[l] \\
\open \ar[ur] \ar[r] & \closed \ar[u] \ar@/^3pt/[ur]& \open \ar[u] &
}
\end{equation}

We assume that new information enters the network through the \emph{source} nodes $\circ$ (i.e., those nodes with no incoming channels), while each non-source node $\closed$ simply collates all of the information that it receives through its incoming channels and transmits the collated information along each of its outgoing channels. 

We might model the distribution of information throughout the network by assigning to each vertex $v$ a set $\functor{D}_v$ containing the information known at the node $v$. The assumption that the non-source nodes collate and pass on the information that they receive is then expressed by the equality
\begin{equation}\label{eq:intro-union}
\functor{D}_v = \bigcup_{\substack{\text{$e$ an edge}\\ \text{with target $v$}}} \functor{D}_{\text{source}(e)},
\end{equation}
for each non-source vertex $v$. Alternatively, we might suppose that at each non-source node $v$ we want to distinguish between copies of the same piece of information arriving through different channels (perhaps to allow multiple sources to confirm one another), and in this case we would replace the union in \eqref{eq:intro-union} with a disjoint union:
\begin{equation}\label{eq:intro-disjoint-union}
\functor{D}_v = \bigsqcup_{\substack{\text{$e$ an edge}\\ \text{with target $v$}}} \functor{D}_{\text{source}(e)}.
\end{equation}

The conditions \eqref{eq:intro-union} and \eqref{eq:intro-disjoint-union} admit a natural generalisation to a more abstract setting, in which the information in the network, and the means by which that information is communicated from one node to another, are modelled respectively by the objects and the morphisms in an arbitrary category. If $G$ is a directed graph (eg, \eqref{eq:intro-graph}), and if $\category{C}$ is a category, then a \emph{diagram in $\category{C}$ of shape $G$} consists of a collection of objects $\functor{D}_v$ of $\category{C}$, one object for each vertex $v$ in the graph $G$; and a collection of morphisms $\functor{D}_e : \functor{D}_{\text{source(e)}} \to \functor{D}_{\text{target}(e)}$ in $\category{C}$, one for each directed edge $e$ in the graph $G$. We say that such a diagram satisfies the \emph{coproduct condition} if for each vertex $v$ that is not a source, the collection of morphisms
\[
\left( \functor{D}_e\ |\ e\text{ is an edge with target $v$}\right)
\]
makes $\functor{D}_v$ into a \emph{coproduct}, in the category $\category{C}$, of the family of objects 
\[
\left( \functor{D}_{\text{source}(e)}\ |\ e\text{ is an edge with target $v$}\right).
\] 
If we let $X$ be the set of all possible pieces of information in our hypothetical network, and we let $\category{C}$ be the category associated with the partially ordered set  of subsets of $X$ under inclusion, then the coproduct condition becomes the condition \eqref{eq:intro-union}. If we let $\category{C}$ be the category of sets and functions, then the coproduct condition becomes the condition \eqref{eq:intro-disjoint-union}. We denote by $\CKL_{\category{C}}(G)$ the category of diagrams in $\category{C}$ of shape $G$ satisfying the coproduct condition, with the natural notion of morphism of diagrams. (A formal definition, along with reminders about the categorical terminology, will be given in Section \ref{sec:def}.) 

The starting point of this paper is the observation that this intuitive picture of information moving through a network can be used to describe modules over \emph{Leavitt path algebras} (cf. \cite{Abrams-book}): if we take $\category{C}=\Mod(\mathbb{F})$ to be the category of vector spaces over some field $\mathbb{F}$, with linear maps as morphisms, and if no vertex in the graph $G$ is the target of infinitely many edges, then we have an equivalence of categories
\begin{equation}\label{eq:intro-LPA-equivalence}
\CKL_{\Mod(\mathbb{F})}(G) \cong \Mod(L_{\mathbb{F}}(G^{\opp}))
\end{equation}
where the right-hand side is the category of modules over the Leavitt path algebra $L_{\mathbb{F}}(G^{\opp})$ of the opposite graph $G^{\opp}$ (i.e., directed graph obtained from $G$ by reversing the orientation of each edge). A proof of this equivalence was given in  \cite{Koc}, where it was noted that the finiteness condition on $G$ can be dropped by modifying the coproduct condition on diagrams. In Theorem \ref{thm:LPA-equivalence} we prove such a generalisation. (Note that in this paper we consider only ordinary directed graphs, rather than the separated graphs considered in \cite{Koc}.) This result, which has surely been discovered or intuited in one form or another by many who work with Leavitt path algebras, gives a picture of modules over Leavitt path algebras that is easy to visualise and axiomatically very simple---arguably simpler than the definition of the algebras themselves in terms of generators and relations.

The equivalence \eqref{eq:intro-LPA-equivalence} is significant because many interesting questions about Leavitt path algebras are concerned with \emph{Morita equivalences} between these algebras, meaning that they are concerned with the categories $\Mod(L_{\mathbb{F}}(G^{\opp}))$, rather than with the algebras per se. Considering the left-hand side of \eqref{eq:intro-LPA-equivalence} reveals a way to generalise these categories that is not immediately apparent when considering only the right-hand side: namely, one can replace $\Mod(\mathbb{F})$ not just with $\Mod(R)$ for an arbitrary commutative ring $R$ (as is done in \cite{Tomforde:R}, for example), but with any category one chooses. The main results of this paper show that many of the known Morita equivalence results for Leavitt path algebras extend to the far more general setting of the diagram categories $\CKL_{\category{C}}(G)$, meaning that these results are `really' theorems about coproducts, rather than about Leavitt path algebras.

To be more precise, our results generalise earlier work (principally results  of \cite[Section 3]{Abrams-FE}, which are modelled on results of \cite{Bates-Pask} for graph $C^*$-algbras) that relate \emph{flow equivalence} of graphs to Morita equivalences of Leavitt path algebras. Theorems \ref{thm:sinks}, \ref{thm:out-delay}, \ref{thm:heads}, \ref{thm:in-delay}, \ref{thm:out-split}, and \ref{thm:in-split} assert that if $G'$ is a directed graph obtained from a graph $G$ by one of several procedures (removing a sink, adding a head to a source, out- or in-delay, or out- or in-splitting, sometimes under additional finiteness conditions) then for each category $\category{C}$ (sometimes required to admit certain coproducts) one has an equivalence of categories $\CKL_{\category{C}}(G) \cong \CKL_{\category{C}}(G')$. 
A notable exception to these results is the desingularisation procedure introduced in the context of graph $C^*$-algebras by Drinen and Tomforde  \cite{Drinen-Tomforde} and later for Leavitt path algebras by \cite{Abrams-desingularisation}. In Section \ref{subsec:desingularisation} we give a simple example to show that the category $\CKL_{\category{C}}(G)$ is in general not invariant under desingularisation.

Another construction on directed graphs that has been much studied in connection with Leavitt path algebras is the \emph{Cuntz splice}; see for instance the discussion in \cite[Section 2]{Abrams-FE}. A long-standing open problem related to this construction asks whether the Leavitt path algebras of the graphs
\begin{equation}\label{eq:intro-Cuntz}
G=\xygraph{
{\closed}="v_1"(:@(ul,ur)"v_1":@(dr,dl)"v_1")
}
\qquad \text{and} \qquad
H=\xygraph{
{\closed}="v_1"(:@(ul,ur)"v_1":@(dr,dl)"v_1":@/^5pt/[r]{\closed}="v_2"
:@(ul,ur)"v_2":@/^5pt/[r]{\closed}="v_3":@(ul,ur)"v_3":@/^5pt/"v_2":@/^5pt/"v_1")
}
\end{equation}
are Morita equivalent. The analogous question for graph $C^*$-algebras is known  to have a positive answer, thanks to \cite{Rordam}, but the proof relies on analytic techniques that are not available in the algebraic setting. The question for Leavitt path algebras can be rephrased as asking whether the categories $\CKL_{\Mod(\mathbb{F})}(G)$ and $\CKL_{\Mod(\mathbb{F})}(H)$ are equivalent, and put in these terms the question admits some obvious generalisations and analogues:
\begin{question}\label{q:Cuntz}
For the graphs $G$ and $H$ of \eqref{eq:intro-Cuntz},
\begin{enumerate}[\rm(1)]
\item Are  $\CKL_{\category{C}}(G)$ and $\CKL_{\category{C}}(H)$ equivalent for every category $\category{C}$? 
\item Are $\CKL_{\category{C}}(G)$ and $\CKL_{\category{C}}(H)$ equivalent for every category $\category{C}$ with binary coproducts? 
\item Are $\CKL_{\category{C}}(G)$ and $\CKL_{\category{C}}(H)$ equivalent for some specific choices of category $\category{C}$? 
\end{enumerate}
\end{question}
As of now we do not have many answers to these questions, besides some examples of categories $\category{C}$ giving a positive answer to question (3); see Example \ref{ex:Cuntz-P}. Negative answers to questions (1) and/or (2), if found, would have significant consequences for the original question about Morita equivalence of the Leavitt path algebras of $G$ and $H$: for instance, a negative answer to (2) would mean that any Morita equivalence between $L_{\mathbb{F}}(G)$ and $L_{\mathbb{F}}(H)$ must use more linear algebra than just the formation of direct sums of vector spaces, in contrast to the known flow equivalence results.

Question \ref{q:Cuntz} cannot be settled using the flow equivalence theorems established in this paper, since the graphical constructions  appearing in those results all preserve the \emph{Parry-Sullivan number} $\PS(G)$ of a finite directed graph $G$ (\cite{Parry-Sullivan}), while the graphs in \eqref{eq:intro-Cuntz} have different Parry-Sullivan numbers. For this reason it is natural to consider the relationship between the invariant $\PS(G)$ and the categories $\CKL_{\category{C}}(G)$:

\begin{question}\label{q:PS}
Is there a category $\category{C}$ with binary coproducts such that for all irreducible, non-trivial finite directed graphs $G$ and $H$, the existence of an equivalence $\CKL_{\category{C}}(G)\cong \CKL_{\category{C}}(H)$ implies $\PS(G)=\PS(H)$?
\end{question}

See Section \ref{subsec:PSBF} for the terminology. A positive answer to Question \ref{q:PS} would imply a negative answer to part (2) (hence also to part (1)) of Question \ref{q:Cuntz}, since $\PS(G)\neq \PS(H)$.

Another, apparently different linear-algebra-free approach to equivalence theorems for Leavitt path algebras and graph $C^*$-algebras involves realising these algebras as convolution algebras of  topological groupoids (see \cite{KPRR,Steinberg,CFST}), and proving that certain constructions on graphs give rise to equivalences of groupoids and hence to Morita equivalences of graph algebras. See \cite{Clark-Sims}  for an instance of this approach. The precise relationship between the groupoids of \cite{KPRR} and the diagram categories studied here is not yet clear to us.

\subsection*{Contents of the paper} In Section \ref{sec:def} we review some terminology and establish notation related to directed graphs and coproducts, and we define the categories $\CKL_{\category{C}}(G)$ that are the main focus of this paper. In Section \ref{sec:LPA} we state and prove Theorem \ref{thm:LPA-equivalence}, extending the equivalence \eqref{eq:intro-LPA-equivalence} to arbitrary graphs. Section \ref{sec:flow} contains our main results relating moves on graphs to equivalences of diagram categories, and we conclude the paper in Section \ref{subsec:PSBF} with a brief discussion of the invariants of Parry-Sullivan and Bowen-Franks and their relation to our diagram categories.

\section{Terminology and notation}\label{sec:def}

\subsection{Directed graphs} 

\begin{definition}
A \emph{directed graph} $G=(G^0,G^1,\source,\target)$ consists of a nonempty set $G^0$ of \emph{vertices}; a set $G^1$ of \emph{edges}; and two maps of sets $\source, \target : G^1\to G^0$ describing, respectively, the \emph{source} and the \emph{target} of each edge. We visualise these data as in \eqref{eq:intro-graph} by representing each vertex as a point, and each edge $e\in G^1$ as an arrow pointing from the vertex $\source(e)$ to the vertex $\target(e)$.

A vertex $v\in G^0$ with $\target^{-1}(v)=\emptyset$ (that is, a vertex that is not the target of any edge) is called a \emph{source}. A vertex $v\in G^0$ with $\# \target^{-1}(v)=\infty$ (that is, a vertex that is the target of infinitely many edges) is called an \emph{infinite receiver}. A vertex $v\in G^0$ with $\source^{-1}(v)=\emptyset$ (that is, a vertex that is not the source of any edge) is called a \emph{sink}.
\end{definition}

\subsection{Coproducts}\label{subsec:coproducts}

We shall use the basic language of categories and functors,  for which one can consult \cite{MacLane} or \cite{Riehl}, for instance. The notion of \emph{coproducts} in a category will play a central role in this paper;  let us briefly review the definition and establish some notation.

\begin{definition}
Let $\category{C}$ be a category and let $(X_i\ |\ i\in I)$ be a  collection of objects of $\category{C}$ indexed by a nonempty set $I$. (We use  $(\ |\ )$ instead of $\{\ |\ \}$ to indicate that by a `collection' or a `family' indexed by $I$ we mean a function $i\mapsto X_i$ rather than the image of that function.) A \emph{coproduct} of $(X_i\ |\ i\in I)$ in $\category{C}$ is an object $X$ of $\category{C}$ together with a collection of morphisms $(\phi_i : X_i \to X\ |\ i\in I)$ in $\category{C}$, having the property that whenever $Y$ is an object of $\category{C}$ equipped with a family of morphisms $(\psi_i : X_i \to Y\ |\ i\in I)$, there is a unique morphism $\psi:X\to Y$ satisfying $\psi\circ \phi_i = \psi_i$ for each $i$. We say that $\category{C}$ \emph{has nonempty coproducts} if every nonempty collection of objects of $\category{C}$ has a coproduct. We say that $\category{C}$ \emph{has binary coproducts} if every pair of objects $(X_1,X_2)$ of $\category{C}$ has a coproduct (equivalently, every collection of objects of $\category{C}$ indexed by a finite nonempty set $I$ has a coproduct). In this paper we will have no need to consider coproducts of empty families.

If a collection of objects $(X_i\ |\ i\in I)$ in $\category{C}$ has a coproduct then we choose one coproduct and denote it by $\bigsqcup_{i\in I} X_i$:  implicit in this notation is both an object $\bigsqcup_{i\in I} X_i$ of $\category{C}$, and a collection of morphisms $X_i\to \bigsqcup_{i\in I} X_i$ which we will refer to as the canonical morphisms into the coproduct. The universal property of coproducts ensures that any two coproducts of the same family of objects are canonically isomorphic, so we don't need to worry about choosing the coproducts in a coherent way. If $(\psi_i: X_i\to Y\ |\ i\in I)$ is a family of morphisms then we write $\bigsqcup_{i\in I} \psi_i : \bigsqcup_{i\in I} X_i \to Y$ for the induced morphism $\psi$ appearing in the definition of coproducts. More generally, if $(Y_j\ |\ j\in J)$ is another family of objects in $\category{C}$ having a coproduct, and if for each $i\in I$ we have a morphism $\psi_i: X_i \to Y_{j_i}$ for some $j_i\in J$, then the composite morphisms $X_i \xrightarrow{\psi_i} Y_{j_i} \xrightarrow{\text{canonical}} \bigsqcup_{j\in J} Y_j$ induce a morphism $\bigsqcup_{i\in I} X_i \to \bigsqcup_{j\in J} Y_j$, which we shall also denote by $\bigsqcup_{i\in I} \psi_i$.
\end{definition}

\begin{examples}
\begin{enumerate}[\rm(1)]
\item If the set $I$ contains a single element $i$ then we can take $\bigsqcup_{i\in I}X_i = X_i$ equipped with the identity map $X_i\to \bigsqcup_{i\in I}X_i$.
\item The category $\category{Set}$ of sets and functions has nonempty coproducts: we can take $\bigsqcup_{i\in I} X_i$ to be the disjoint union of the sets $X_i$, with the canonical inclusions $X_i\to \bigsqcup_{i\in I} X_i$. 
\item Let $(P,\leq)$ be a partially ordered set, and let $\category{P}$ denote the category whose objects are the elements of the set $P$, such that for each pair of objects $x,y\in P$ there is either one morphism $x\to y$ (if $x\leq y$) or no such morphisms (if $x\not\leq y$). An object $y\in P$ is a coproduct of a nonempty collection of objects $(x_i\ |\ i\in I)$ if and only if $y$ is the supremum of the set $\{x_i\ |\ i\in I\}$. Thus $\category{P}$ has binary coproducts if and only if every  nonempty finite subset of ${P}$ has a supremum, while $\category{P}$ has nonempty coproducts if every nonempty subset has a supremum.
\item The category $\Mod(R)$ of left modules over a ring $R$ with identity has nonempty coproducts: we can take $\bigsqcup_{i\in I} M_i = \bigoplus_{i\in I} M_i$, the direct sum of the $R$-modules $M_i$, equipped with the canonical $R$-module embeddings $M_i\into \bigoplus_{i\in I} M_i$. 
\end{enumerate}
\end{examples}

\subsection{Categories of diagrams}

If $G=(G^0,G^1,\source,\target)$ is a directed graph, and $\category{C}$ is a category, then a \emph{diagram} $\functor{D}$ in $\category{C}$ of shape $G$ consists of an object $\functor{D}_v$ of $\category{C}$ for each vertex $v\in G^0$; and a morphism $\functor{D}_e : \functor{D}_{\source(e)}\to \functor{D}_{\target(e)}$ in $\category{C}$ for each edge $e\in G^1$. A morphism of diagrams $T:\functor{D}\to \functor{E}$ consists of a morphism $T_v:\functor{D}_v\to \functor{E}_v$ in $\category{C}$ for each vertex $v\in G^0$, such that for each edge $e\in G^1$ one has $T_{\target(e)}\circ \functor{D}_e = \functor{D}_e\circ T_{\source(e)}$. We let $\Diag_{\category{C}}(G)$ denote the category of diagrams in $\category{C}$ of shape $G$.

\begin{definition}\label{def:CKL}
Let $G$ be a directed graph and let $\category{C}$ be a category. A diagram $\functor{D}$ in $\category{C}$ of shape $G$ is said to satisfy the \emph{coproduct condition} if for each vertex $v\in G^0$ that is not a source, the collection of morphisms
\[
\left( \left. \functor{D}_e: \functor{D}_{\source(e)} \to \functor{D}_v\ \right|\ e\in \target^{-1}(v) \right )
\]
makes $\functor{D}_v$ a coproduct, in $\category{C}$, of the collection of objects
\[
\left( \left. \functor{D}_{\source(e)}\ \right|\ e\in \target^{-1}(v) \right).
\]
We often express this condition by saying that the map
\[
\bigsqcup_{e\in \target^{-1}(v)} \functor{D}_{e} : \bigsqcup_{e\in \target^{-1}(v)} \functor{D}_{\source(e)}\to \functor{D}_v
\]
is an isomorphism, where this  includes the requirement that the coproduct of the $\functor{D}_{\source(e)}$s should exist. We denote by $\CKL_{\category{C}}(G)$ the category whose objects are the diagrams in $\category{C}$ of shape $G$ satisfying the coproduct condition, and whose morphisms are the morphisms of diagrams defined above.
\end{definition}

\begin{example}\label{ex:no-edges}
If $G$ has no edges then the coproduct condition is vacuously satisfied by all diagrams of shape $G$, and the category $\CKL_{\category{C}}(G)$ is equivalent in an obvious way to the product category $\category{C}^{\#G^0}$.
\end{example}

\begin{example}\label{ex:one-edge}
If a vertex $v\in G^0$ is the target of exactly one edge $e$, then the coproduct condition at $v$ is the condition that  $\functor{D}_e : \functor{D}_{\source(e)} \to \functor{D}_v$ should be an isomorphism. So, for example, if $G$ is the graph $\xygraph{\bullet="v":@(ur,dr)"v"}$ then $\CKL_{\category{C}}(G)$ is the category whose objects are pairs $(X,\phi)$, where $X$ is an object of $\category{C}$ and $\phi$ is an automorphism of $X$; and whose morphisms $(X,\phi)\to (Y,\psi)$ are those morphisms $\rho:X\to Y$ in $\category{C}$ that satisfy $\psi\circ \rho = \rho \circ \phi$.
\end{example}

\begin{example}
When $\category{C}=\category{Set}$ the coproduct condition  is the requirement that if $v\in G^0$ is not a source then the maps $\functor{D}_e : \functor{D}_{\source(e)} \to \functor{D}_v$ induce a bijection $\bigsqcup_{e\in \target^{-1}(v)} \functor{D}_{\source(e)} \xrightarrow{\cong} \functor{D}_v$. If $G$ has no infinite receivers then an object of $\CKL_{\category{Set}}(G)$ is essentially the same thing as a \emph{$G^{\opp}$-algebraic branching system} as studied in \cite{Goncalves}. (If $G$ does have infinite receivers then the $G^{\opp}$-branching systems are the objects of $\CKL_{\category{Set}}(G_+)$, where $G_+$ is defined in Definition \ref{def:plus} below.)
\end{example}

\begin{example}\label{ex:CKL_P}
Let $\category{P}$ be the category corresponding to a partially ordered set $(P,\leq)$. A diagram $\functor{D}$ in $\category{P}$ of shape $G$ is the same thing as a function $\functor{D}:G^0\to P$ with the property that $\functor{D}_{\source(e)}\leq \functor{D}_{\target(e)}$ for each $e\in G^1$. The collection of all such diagrams is itself a partially ordered set, with  $\functor{D}\leq \functor{E}$ if and only if $\functor{D}_v \leq \functor{E}_v$ for every $v\in G^0$. The coproduct condition in this case is the requirement that if $v$ is not a source then $\functor{D}_v = \sup\{ \functor{D}_{\source(e)}\ |\ e\in \target^{-1}(v)\}$, and it is clear from this description that the partially ordered set $\CKL_{\category{P}}(G)$ does not depend on the number of edges in $G$ having a given source and target, but only on whether or not that number is zero. 
\end{example}

\begin{example}\label{example:Mod(R)}
When $\category{C}= \Mod(R)$ the coproduct condition is the requirement that if $v$ is not a source then the maps $\functor{D}_e : \functor{D}_{\source(e)} \to \functor{D}_v$ induce an isomorphism of $R$-modules $\bigoplus_{e\in \target^{-1}(v)} \functor{D}_{\source(e)} \xrightarrow{\cong} \functor{D}_v$. It was proved in \cite[Proposition 3.2]{Koc} that if $G$ has no infinite receivers then for each field $\mathbb{F}$ the category $\CKL_{\Mod(\mathbb{F})}(G)$ is equivalent to the category of modules over the Leavitt path algebra $L_{\mathbb{F}}(G^{\opp})$. It was also noted there that the finiteness hypothesis can be relaxed at the cost of introducing a slightly more complicated coproduct condition. We shall present a generalisation of this sort in Theorem \ref{thm:LPA-equivalence}.
\end{example}

\subsection{Functoriality in $\category{C}$}

For a fixed directed graph $G$, the passage from a category $\category{C}$ to the diagram category $\Diag_{\category{C}}(G)$ is functorial:  if $\functor{F}:\category{C}\to \category{D}$ is a functor, then we obtain a functor $\Diag_{\category{C}}(G)\to \Diag_{\category{D}}(G)$ by applying $\functor{F}$ to each object and morphism in a diagram. If $\functor{F}$ preserves coproducts then the induced functor on diagrams preserves the coproduct condition from Definition \ref{def:CKL}, and hence induces a functor
\[
\CKL_{\functor{F}}(G) : \CKL_{\category{C}}(G)\to \CKL_{\category{D}}(G).
\]
An example of this construction is already well known in the literature:

\begin{example}
Let $R$ be a commutative ring with identity, and let $\functor{F}:\category{Set}\to \Mod(R)$ be the functor that sends each set $X$ to the free $R$-module on $X$, and each map of sets to the induced map on free modules. Since $\functor{F}$ sends disjoint unions of sets to direct sums of free modules, it preserves coproducts and thus induces a functor from $\CKL_{\category{Set}}(G)$ to $\CKL_{\Mod(R)}(G)$. This functor is the familiar construction of modules over Leavitt path algebras from algebraic branching systems described in \cite{Goncalves} and applied in  \cite{Chen}, among other places.
\end{example}

The functoriality of $\CKL_{\category{C}}(G)$ can be expressed more comprehensively using the language of $2$-categories: the assignments $\category{C}\mapsto \CKL_{\category{C}}(G)$ and $\functor{F}\mapsto \CKL_{\functor{F}}(G)$ extend to a \emph{$2$-functor} from the $2$-category of categories, coproduct-preserving functors, and natural transformations to the $2$-category of categories, functors, and natural transformations. (See \cite{Johnson-Yau} for the language.) The equivalences of categories $\CKL_{\category{C}}(G)\cong \CKL_{\category{C}}(H)$ that we exhibit in Theorems \ref{thm:sinks}, \ref{thm:out-delay}, \ref{thm:heads}, \ref{thm:in-delay}, \ref{thm:out-split}, and \ref{thm:in-split} can likewise be promoted to equivalences of $2$-functors, meaning roughly that these equivalences are natural with respect to $\category{C}$. This naturality is in fact very easy to see once one knows what to look for, but in order to avoid overloading the paper with definitions we have decided against going into the $2$-categorical details here.

\section{Modules over Leavitt path algebras as diagrams}\label{sec:LPA}

In this section we formulate and prove the generalisation of \cite[Proposition 3.2]{Koc} promised in Example \ref{example:Mod(R)}. 

First let us recall the definition of Leavitt path algebras, for which our main reference will be \cite{Abrams-book}. In the graph $C^*$-algebra literature one finds two distinct conventions regarding the defining relations in graph algebras, and of these two the convention used in \cite{Raeburn} (where $e=\target(e)e\source(e)$ for each edge $e$) is the most convenient for our purposes here. That convention does not, however, seem to have been adopted in the Leavitt path algebra literature. We have attempted in this paper to balance notational convenience with adherence to the established conventions by adopting the definition used in \cite{Abrams-book} (and in all of the other papers on Leavitt path algebras listed in our bibliography), but  always applying that definition to the \emph{opposite} graph $G^{\opp}$ of a directed graph $G$ (that is, the graph with the same vertices and edges as $G$, but with the orientation of each edge reversed). 

\begin{definition}\label{def:LPA}
Let $G=(G^0,G^1,\source,\target)$ be a directed graph, and let $R$ be a commutative ring with identity. The \emph{Leavitt path algebra} $L_R(G^{\opp})$ of the opposite graph $G^{\opp}=(G^0,G^1,\target,\source)$ is the universal $R$-algebra with generators $G^0 \sqcup G^1 \sqcup (G^1)^*$ (where $(G^1)^* = \{ e^*\ |\ e\in G^1\}$ is just a second copy of $G^1$, distinguished from the first copy by the otherwise meaningless notation $*$) and with the following relations:
\begin{enumerate}[{\rm(1)}]
\item $uv=\delta_{u,v}v$ for all $u,v\in G^0$;
\item $\target(e)e=e \source(e)=e$ for all $e\in G^1$;
\item $\source(e)e^*=e^*\target(e)= e^*$ for all $e\in G^1$;
\item $e^* f = \delta_{e,f} \source(e)$ for all $e,f\in G^1$; and
\item $v = \sum_{\substack{e\in  \target^{-1}(v)}} ee^*$ for all  $v\in G^0$ for which the set $\target^{-1}(v)$ is finite and non-empty.
\end{enumerate}
\end{definition}

The fact that the relation (5) does not say anything about infinite receivers is the reason why the equivalence \eqref{eq:intro-LPA-equivalence} does not hold for arbitrary graphs. To obtain an equivalence we must allow an extra degree of freedom in the coproduct condition at each infinite receiver, and we will see that this extra degree of freedom can be incorporated very simply by adding an extra vertex to the graph as described in  the next definition.

\begin{definition}\label{def:plus}
Let $G=(G^0,G^1,\source,\target)$ be a directed graph. Define a new directed graph $G_+=(G_+^0,G_+^1,\source_+,\target_+)$ as follows:
\[
\begin{aligned}
& G_+^0 \coloneqq  G^0 \sqcup \{ v_+ \ |\ v\in G^0,\, \#\target^{-1}(v)=\infty\} \\
& G_+^1 \coloneqq G^1 \sqcup \{ e_+\ |\ e\in G^1,\, \#\target^{-1}(\source(e))=\infty \} \\
&\source_+(e) = \source(e),\quad \source_+(e_+)=\source(e)_+ \\
& \target_+(e) = \target(e), \quad \target_+(e_+)=\target(e).
\end{aligned}
\] 
In words, $G_+$ is obtained from $G$ by adding an extra vertex $v_+$ for each infinite receiver $v$; and for each each edge $e$ whose source is an infinite receiver, adding a new edge $e_+$ from $\source(e)_+$ to $\target(e)$. Note that all of the vertices $v_+$ are sources in $G_+$. Note too that if $G$ has no infinite receivers then $G_+=G$.
\end{definition}

Before stating the main result of this section we need one more piece of terminology. If $A$ is a ring that does not necessarily have a multiplicative identity then we let $\Mod(A)$ be the category whose objects are those left $A$-modules $M$ with the property that
\[
M = AM \coloneqq \{\text{finite sums of elements of the form $am$ with $a\in A$ and $m\in M$}\},
\]
and whose morphisms are the $A$-linear maps. For Leavitt path algebras (and more generally, for rings admitting \emph{local units}), equivalences between the categories $\Mod(A)$ has been shown to be a sensible and useful notion of Morita equivalence; see e.g. \cite{Abrams-Morita}.

\begin{theorem}\label{thm:LPA-equivalence}
Let $G$ be a directed graph and let $R$ be a commutative ring with identity. There is an equivalence of categories $\Mod(L_R(G^{\opp}))\cong \CKL_{\Mod(R)}\left( G_+\right)$.
\end{theorem}

\begin{proof}
Let $\functor{D}$ be an object of $\CKL_{\Mod(R)}(G_+)$. For each $v\in G^0$ we have an $R$-module $\functor{D}_v$, and if $v$ receives infinitely many edges in $G$ then we also have a second $R$-module $\functor{D}_{v_+}$. Likewise, for each edge $e\in G^1$ we have an $R$-linear map $\functor{D}_e : \functor{D}_{\source(e)}\to \functor{D}_{\target(e)}$, and if $\source(e)$ receives infinitely many edges in $G$ then we have a second $R$-linear map $\functor{D}_{e_+}: \functor{D}_{\source(e)_+} \to \functor{D}_{\target(e)}$. To reduce the necessity to distinguish between the infinite receivers and non-infinite receivers, it is convenient to define $\functor{D}_{v_+}=0$ for all vertices $v\in G^0$ that do not receive infinitely many edges, and to define $\functor{D}_{e_+}$ to be the zero map $0: \functor{D}_{\source(e)_+}\to \functor{D}_{\target(e)}$ for all edges $e\in G^1$ for which $\source(e)$ does not receive infinitely many edges.

The coproduct condition on $\functor{D}$ means that for each vertex $v\in G^0$ that is not a source, the map
\begin{equation}\label{eq:M-directsum-iso}
\bigoplus_{e\in\target^{-1}(v)} (\functor{D}_e \oplus \functor{D}_{e_+}) : \bigoplus_{e\in \target^{-1}(v)} (\functor{D}_{\source(e)}\oplus \functor{D}_{\source(e)_+}) \to \functor{D}_v
\end{equation}
is  an isomorphism of $R$-modules. We let $\phi_v$ denote the inverse of this isomorphism, and for each $e\in G^1$ we let $\functor{D}_e^* : \functor{D}_{\target(e)}\to \functor{D}_{\source(e)}$ be the composition
\[
\functor{D}_e^*: \functor{D}_{\target(e)} \xrightarrow{\phi_{\target(e)}} \bigoplus_{f\in \target^{-1}\target(e)} (\functor{D}_{\source(f)}\oplus \functor{D}_{\source(f)_+}) \xrightarrow{\text{project}} \functor{D}_{\source(e)}.
\]
We define $\functor{D}_{e_+}^*:\functor{D}_{\target(e)}\to \functor{D}_{\source(e)_+}$ in a similar way, but projecting onto $\functor{D}_{\source(e)_+}$ instead of $\functor{D}_{\source(e)}$.

Now we associate to $\functor{D}$ a left module $M=\Phi(\functor{D})$ over $L_{R}(G^{\opp})$ as follows. We first consider the $R$-module
\[
M \coloneqq \bigoplus_{v\in G^0} \functor{D}_v \oplus \functor{D}_{v_+}.
\]
For each $v\in G^0$ we let $P_v : M\to M$ be map that is the identity on the summand $\functor{D}_v\oplus \functor{D}_{v_+}$ and zero on the summands $\functor{D}_w\oplus \functor{D}_{w_+}$ for $w\neq v$. For each $e\in G^1$ we let $A_e : M\to M$ be the $R$-linear map equal to $\functor{D}_e \oplus \functor{D}_{e_+} : \functor{D}_{\source(e)}\oplus \functor{D}_{\source(e)_+} \to \functor{D}_{\target(e)}$ on the direct summand $\functor{D}_{\source(e)}\oplus \functor{D}_{\source(e)_+}$ of $M$, and equal to zero on the other summands. For each $e\in G^1$ we also let $A_e^* : M\to M$ be the $R$-linear map equal to $\functor{D}_e^* \oplus \functor{D}_{e_+}^* : \functor{D}_{\target(e)} \to \functor{D}_{\source(e)}\oplus \functor{D}_{\source(e)_+}$ on the direct summand $\functor{D}_{\target(e)}$ of $M$, and equal to zero on the other direct summands.

We claim that the linear endomorphisms $P_v$,  $A_e$, and $A_e^*$ of the $R$-module $M$ satisfy the relations (1)--(5) of Definition \ref{def:LPA}; that is, we claim that those relations become equalities in $\End_R(M)$ upon replacing each $v$ by $P_v$, each $e$ by $A_e$, and each $e^*$ by $A_e^*$.

The relation (1) follows immediately from the definitions, since if $u\neq v$ then $P_u$ and $P_v$ project onto distinct direct summands in $M$. Relation (2) is  also a straightforward consequences of the definitions, since $A_e$ is supported on the image of $P_{\source(e)}$ and has image contained in the image of $P_{\target(e)}$, and both $P_{\source(e)}$ and $P_{\target(e)}$ act as the identity on their image. Similar considerations applied to $A_e^*$ show that relation (3) also holds.

Turning to (4), we first note that (1), (2), and (3) ensure that $A_e^* A_f =0$ unless $\target(e)=\target(f)$. Assuming that $\target(e)=\target(f)=v$, the map $A_e^* A_f$ is supported on $\functor{D}_{\source(f)}\oplus \functor{D}_{\source(f)_+}$, has image contained in $\functor{D}_{\source(e)}\oplus \functor{D}_{\source(e)_+}$, and is given by the composition
\begin{equation}\label{eq:FM-CK1}
\functor{D}_{\source(f)}\oplus \functor{D}_{\source(f)_+} \xrightarrow{\functor{D}_f \oplus \functor{D}_{f_+}} \functor{D}_{v} \xrightarrow{\phi_v} \bigoplus_{g\in\target^{-1}(v)} \functor{D}_{\source(g)} \oplus \functor{D}_{\source(g)_+} \xrightarrow{\text{project}} \functor{D}_{\source(e)}\oplus \functor{D}_{\source(e)_+}.
\end{equation}
The definition of $\phi_v$ ensures that the composition $\phi_v \circ (\functor{D}_{f}\oplus \functor{D}_{f_+})$ is the canonical embedding of $\functor{D}_{\source(f)}\oplus \functor{D}_{\source(f)_+}$ into $\bigoplus_{g\in \target^{-1}(v)} \functor{D}_{\source(g)}\oplus \functor{D}_{\source(g)_+}$, and this makes it clear that the composition \eqref{eq:FM-CK1} is the identity if $e=f$, and zero otherwise. Thus the relation (4) holds.

Finally considering the relation (5), let $v\in G^0$ be a vertex that is the target of a nonzero, finite number of edges in $G$. Note that we then have $\functor{D}_{v_+}=0$. The $R$-linear map $\sum_{e\in \target^{-1}(v)} A_e A_e^*:M\to M$ thus has support and image contained in the direct summand $\functor{D}_{v}$, and the definitions of $A_e$ and of $A_e^*$ show immediately that this map is the composition
\[
\functor{D}_v \xrightarrow{\phi_v} \bigoplus_{e\in \target^{-1}(v)} \functor{D}_{\source(e)}\oplus \functor{D}_{\source(e)_+} \xrightarrow{\bigoplus_{e\in \target^{-1}(v)} \functor{D}_e\oplus \functor{D}_e^*} \functor{D}_v,
\]
which is the identity on $\functor{D}_v$. Thus the relation (5) holds.

Now the universal property of the Leavitt path algebra $L_R(G^{\opp})$ gives a homomorphism $L_R(G^{\opp})\to \End_R(M)$, defined uniquely by $v\mapsto P_v$, $e\mapsto A_e$, and $e^*\mapsto A_e^*$ for $v\in G^0$ and $e\in G^1$, and we use this homomorphism to regard $M$ as an $L_R(G^{\opp})$-module. To see that this module satisfies $L_R(G^{\opp})M=M$, and hence is an object of $\Mod(L_R(G^{\opp}))$, note that $vm=m$ for all $m$ in the summand $\functor{D}_v\oplus \functor{D}_{v_+}$ of $M$.

To turn this construction $M\mapsto \Phi(M)$ into a functor from the category $\CKL_{\Mod(R)}(G_+)$ to $\Mod(L_R(G^{\opp}))$, we now define the action of $\Phi$ on morphisms. Given a morphism $T : \functor{D}\to \functor{E}$ in $\CKL_{\Mod(R)}(G_+)$, meaning a collection of $R$-linear maps $T_v : \functor{D}_v\to \functor{E}_v$ and $T_{v_+} : \functor{D}_{v_+} \to \functor{E}_{v_+}$ satisfying $T_{\target(e)} \functor{D}_e = \functor{E}_e T_{\source(e)}$ and $T_{\target(e)} \functor{D}_{e_+} = \functor{E}_{e_+} T_{\source(e)_+}$ for all $e\in G^1$, we define $\Phi(T) : \Phi(\functor{D})\to \Phi(\functor{E})$ to be the direct sum 
\[
\Phi(T)\coloneqq \bigoplus_{v\in G^0} T_v\oplus T_{v_+} : \bigoplus_{v\in G^0} \functor{D}_v\oplus \functor{D}_{v_+} \to \bigoplus_{v\in G^0} \functor{E}_v \oplus \functor{E}_{v_+}.
\]

We clearly have $\Phi(T)v = v\Phi(T)$ and $\Phi(T)e = e\Phi(T)$ for all $v\in G^0$ and all $e\in G^1$. To see that we also have $\Phi(T)e^* = e^* \Phi(T)$, first note that both sides have support in the direct summand $\functor{D}_{\target(e)}$ of $M$, and that the restrictions of these maps to this summand are equal to $(T_{\source(e)}\oplus T_{\source(e)_+})\circ(\functor{D}_{e}^*\oplus \functor{D}_{e_+}^*)$ and $(\functor{E}_e^* \oplus \functor{E}_{e_+}^*)\circ T_{\target(e)}$, respectively. Recalling the isomorphism \eqref{eq:M-directsum-iso}, we see that in order to prove that these two maps are equal, it will suffice to prove that 
\begin{equation}\label{eq:TFestar}
(T_{\source(e)}\oplus T_{\source(e)_+})\circ(\functor{D}_{e}^*\oplus \functor{D}_{e_+}^*) \circ (\functor{D}_f\oplus \functor{D}_{f_+})=(\functor{E}_e^* \oplus \functor{E}_{e_+}^*)\circ T_{\target(e)} \circ (\functor{D}_f\oplus \functor{D}_{f_+})
\end{equation}
for each $f\in \target^{-1}\target(e)$. The fact that the maps  $\functor{D}_{e}^*\oplus \functor{D}_{e_+}^*$ and $\functor{D}_f\oplus \functor{D}_{f_+}$ satisfy relation (4) from Definition \ref{def:LPA} ensures that the left-hand side in \eqref{eq:TFestar} is equal to $\delta_{e,f}(T_{\source(e)}\oplus T_{\source(e)_+})$. On the other hand, the right-hand side of \eqref{eq:TFestar} is
\[
\begin{aligned}
(\functor{E}_e^* \oplus \functor{E}_{e_+}^*)\circ T_{\target(e)} \circ (\functor{D}_f\oplus \functor{D}_{f_+}) & = (\functor{E}_e^* \oplus \functor{E}_{e_+}^*)(\functor{E}_f \oplus \functor{E}_{f_+})(T_{\source(f)}\oplus T_{\source(f)_+}) \\ & = \delta_{e,f} (T_{\source(f)}\oplus T_{\source(f)_+}),
\end{aligned}
\]
where the first equality holds because $T$ is a morphism of diagrams, and the second equality follows from the fact that $\functor{E}_{e}^*\oplus \functor{E}_{e_+}^*$ and $\functor{E}_f\oplus \functor{E}_{f_+}$ satisfy relation (4) from Definition \ref{def:LPA}. This establishes the equality \eqref{eq:TFestar}, and shows that our map $\Phi(T):\Phi(\functor{D})\to \Phi(\functor{E})$ is a map of $L_R(G^{\opp})$-modules.  It is clear from the definition that $\Phi$ preserves composition and identities, and so we have defined a functor $\Phi: \CKL_{\Mod(R)}(G_+)\to \Mod(L_R(G^{\opp}))$.

To show that $\Phi$ is an equivalence we must check that it is full (surjective on morphisms), faithful (injective on morphisms), and essentially surjective (surjective, up to isomorphism, on objects). To see that $\Phi$ is full, let $S : \Phi(\functor{D})\to \Phi(\functor{E})$ be a morphism of $L_R(G^{\opp})$-modules, for diagrams $\functor{D}$ and $\functor{E}$ in $\CKL_{\Mod(R)}(G_+)$. For each $v\in G^0$ we let $T_v : \functor{D}_v\to \functor{E}_v$ be the composition
\[
\functor{D}_v \xrightarrow{\text{inclusion}} \bigoplus_{u\in G^0} (\functor{D}_u\oplus \functor{D}_{u_+}) = \Phi(\functor{D}) \xrightarrow{S} \Phi(\functor{E}) = \bigoplus_{u\in G^0} (\functor{E}_u\oplus \functor{E}_{u_+}) \xrightarrow{\text{project}} \functor{E}_v.
\]
We define $T_{v_+}:\functor{D}_{v_+}\to \functor{E}_{v_+}$ analogously, by including $\functor{D}_{v_+}$ and projecting onto $\functor{E}_{v_+}$. The fact that $Se=eS$ for all $e\in G^1$ ensures that $T_{\target(e)}\functor{D}_e = \functor{E}_e T_{\source(e)}$ and that $T_{\target(e)}\functor{D}_{e_+} =\functor{E}_{e_+} T_{\source(e)_+}$, and so $T$ is a morphism in $\CKL_{\Mod(R)}(G_+)$. Consulting the definition of $\Phi$ on morphisms shows immediately that $\Phi(T)=S$.

To see that $\Phi$ is faithful, suppose that we have two morphisms $T,T':\functor{D}\to \functor{E}$ in $\CKL_{\Mod(R)}(G_+)$ with $\Phi(T)=\Phi(T')$. Then $\bigoplus_{v\in G^0}T_v\oplus T_{v_+}= \bigoplus_{v\in G^0} T'_v\oplus T'_{v_+}$ as maps $\bigoplus_{v\in G^0} \functor{D}_v\oplus \functor{D}_{v_+} \to \bigoplus_{v\in G^0} \functor{E}_v\oplus \functor{E}_{v_+}$, and restricting this equality to the individual summands gives $T_v=T'_v$ and $T_{v_+}=T'_{v_+}$ for each $v\in G^0$, and so $T=T'$ as morphisms of diagrams.

Finally, to show that $\Phi$ is essentially surjective, let $M$ be an $L_R(G^{\opp})$-module satisfying $M=L_R(G^{\opp})M$. For each $v\in G^0$ we define an $R$-linear map $\tilde{v}: M\to M$ as follows: 
\[
\tilde{v}(m) = \begin{cases} vm & \text{if $v$ is not an infinite receiver,}\\  \sum_{e\in \target^{-1}(v)} ee^*m & \text{if $v$ is an infinite receiver}.\end{cases}
\]
(Our assumption $M=L_R(G^{\opp})M$ ensures that for each $m\in M$ the sum on the right-hand side has only finitely many nonzero terms; cf. \cite[Lemma 1.2.12(v)]{Abrams-book}.) We then define $v_+\in \End_R(M)$ by $v_+(m)\coloneqq vm-\tilde{v}(m)$. 

Now we define a diagram $\functor{D}$ in $\Mod(R)$ of shape $G_+$ as follows. For each $v\in G^0$ we set $\functor{D}_v=\tilde{v}M$ and $\functor{D}_{v_+}=v_+M$. For each $e\in G^1$  we define $\functor{D}_e:\functor{D}_{\source(e)}\to \functor{D}_{\target(e)}$ to be the $R$-linear map $\widetilde{\source(e)}M\xrightarrow{m\mapsto em} \widetilde{\target(e)}M$, and we define  $\functor{D}_{e_+} : \functor{D}_{\source(e)_+}\to \functor{D}_{\target(e)}$ to be the map $\source(e)_+ M \xrightarrow{m\mapsto em} \widetilde{\target(e)}M$. For each $v\in G^0$ that is not a source, the relations in Definition \ref{def:LPA} together with the definition of $\tilde{v}$ ensure that the map
\[
\bigoplus_{e\in \target^{-1}(v)} e : \bigoplus_{e\in \target^{-1}(v)} \source(e)M \to \tilde{v} M
\]
is an isomorphism of $R$-modules. Writing $\source(e)M=\widetilde{\source(e)}M\oplus \source(e)_+ M$ and consulting the definitions of $\functor{D}_e$, $\functor{D}_{e_+}$, and $\functor{D}_v$ above, we see that the map
\[
\bigoplus_{\target(e)=v} \functor{D}_e \oplus \functor{D}_{e_+} : \bigoplus_{\target(e)=v} \functor{D}_{\source(e)}\oplus \functor{D}_{\source(e)_+} \to \functor{D}_v
\]
is an isomorphism, and so our diagram $\functor{D}$ satisfies the coproduct condition. Now we have $\Phi(\functor{D}) = \bigoplus_{v\in G^0} \tilde{v}M\oplus v_+ M = \bigoplus_{v\in G^0} vM$, and the summation map $\bigoplus_{v\in G^0} vM \to M$ is easily seen to be an isomorphism of $L_R(G^{\opp})$-modules $\Phi(\functor{D})\to M$. Thus $\Phi$ is essentially surjective, and this completes the proof that $\Phi$ is an equivalence. 
\end{proof}

\section{Flow equivalence}\label{sec:flow}

In this section we will show that many of the known results relating constructions on directed graphs to Morita equivalence of Leavitt path algebras (see, principally, \cite[Section 3]{Abrams-FE}) extend to the more general setting of the categories $\CKL_{\category{C}}(G)$, subject in some cases to finiteness conditions on $G$, and/or a requirement that $\category{C}$ admit certain coproducts. In Sections \ref{subsec:sinks}--\ref{subsec:in-split} we establish equivalence results of this kind, while in section \ref{subsec:desingularisation} we exhibit an example of a graphical construction that is known to yield Morita equivalent Leavitt path algebras, but which does not  lead to equivalences of the categories $\CKL_{\category{C}}$ in general, even when $\category{C}$ has all coproducts. In Section \ref{subsec:PSBF} we conclude the paper with some remarks relating the invariants of Parry-Sullivan and Bowen-Franks to the categories $\CKL_{\category{C}}(G)$.

We remind the reader that when comparing our results to the known theorems for Leavitt path algebras, one should keep in mind that the passage from diagrams to modules over Leavitt path algebras involves replacing $G$ by $G^{\opp}$. So for instance our Theorem \ref{thm:sinks} on sinks generalises \cite[Proposition 3.1]{Abrams-FE}, which concerns sources; our Theorem \ref{thm:out-delay} on out-delays generalises \cite[Theorem 3.6]{Abrams-FE}, which concerns  in-delays; and so on.

\subsection{Sink removal}\label{subsec:sinks}

\begin{definition}
Let $G=(G^0,G^1,\source,\target)$ be a directed graph. Recall that a vertex $w\in G^0$ is called a \emph{sink} if $\source^{-1}(w)=\emptyset$. If $w$ is a sink then we let $G-w$ be the directed graph with $(G-w)^0=G^0\setminus \{w\}$, $(G-w)^1 = G^1 \setminus \target^{-1}(w)$, and with source and target maps restricted from those of $G$.
\end{definition}

In words, $G-w$ is the graph obtained from $G$ by removing the sink $w$ along with all edges having target $w$.

\begin{theorem}\label{thm:sinks}
Let $G$ be a directed graph, and let $w\in G^0$ be a sink that is not a source. For each category $\category{C}$ with nonempty coproducts there is an equivalence of categories 
\(
\CKL_{\category{C}}(G) \cong \CKL_{\category{C}}(G-w).
\)
If in addition $w$ is not an infinite receiver, then the same is true for all categories $\category{C}$ with binary coproducts.
\end{theorem}

\begin{proof}
For each object $\functor{D}$ of $\CKL_{\category{C}}(G)$ we let $\Phi (\functor{D})$ be the object of $\CKL_{\category{C}}(G-w)$ defined by
\[
\Phi (\functor{D})_v \coloneqq \functor{D}_v \quad \text{and}\quad \Phi (\functor{D})_e \coloneqq \functor{D}_e
\]
for all $v\in (G-w)^0$ and all $e\in (G-w)^1$. That is, $\Phi (\functor{D})$ is the diagram obtained from the diagram $\functor{D}$ by deleting the object $\functor{D}_w$ and all morphisms mapping into $\functor{D}_w$. This new diagram clearly inherits the coproduct condition from the diagram $\functor{D}$. Similarly, for each morphism $T:\functor{D}\to \functor{E}$ in $\CKL_{\category{C}}(G)$ we let $\Phi (T):\Phi (\functor{D}) \to \Phi (\functor{E})$ be defined by 
\[
\Phi (T)_v \coloneqq T_v
\]
for all $v\in (G-w)^0$. The fact that $\Phi (T)$ is a morphism of diagrams follows immediately from the corresponding property of $T$, while the fact that $\Phi$ preserves identities and commutes with composition is obvious from the definition. Thus $\Phi $ is a functor. 

To see that $\Phi $ is full and faithful, let $\functor{D}$ and $\functor{E}$ be two diagrams in $\CKL_{\category{C}}(G)$, and let $S:\Phi (\functor{D}) \to \Phi (\functor{E})$ be a morphism in $\CKL_{\category{C}}(G-w)$. Since $w$ is not a source the diagram $\functor{D}$ satisfies the coproduct condition at $w$, and hence there is a unique morphism $T_w : \functor{D}_w \to \functor{E}_w$ making the diagram
\[
\xymatrix{
\functor{D}_{\source(e)} \ar[r]^-{\functor{D}_e} \ar[d]_{S_{\source(e)}} & \functor{D}_w \ar[d]^-{T_w} \\
\functor{E}_{\source(e)} \ar[r]^-{\functor{E}_e} & \functor{E}_w
}
\]
commute for each $e\in \target^{-1}(w)$. Setting $T_v\coloneqq S_v$ for $v\in G^0\setminus \{w\}$ then shows that there is a unique morphism $T:\functor{D}\to \functor{E}$ in $\CKL_{\category{C}}(G)$ with $\Phi (T)=S$.

Finally, to see that $\Phi$ is essentially surjective note that each diagram $\functor{E}$ in $\CKL_{\category{C}}(G-w)$ can be extended to a diagram $\widetilde{\functor{E}}$ in $\CKL_{\category{C}}(G)$ by setting $\widetilde{\functor{E}}_w \coloneqq \bigsqcup_{e\in \target^{-1}(w)} \functor{E}_{\source(e)}$ (which we can do since $\target^{-1}(w)$ is a nonempty set and $\category{C}$ has nonempty coproducts), and then letting $\widetilde{\functor{E}}_e : \functor{E}_{\source(e)} \to \functor{E}_w$ be the canonical mapping into the coproduct for each $e\in \target^{-1}(w)$. Then $\Phi (\widetilde{\functor{E}})=\functor{E}$, so $\Phi $ is an equivalence. If $w$ is not an infinite receiver then the coproduct that we formed when defining $\widetilde{\functor{E}}$ is over a finite set, and so $\category{C}$ need only admit binary coproducts.
\end{proof}

Theorem \ref{thm:sinks} gives the following generalisation of a well-known fact about finite-dimensional Leavitt path algebras (see e.g. \cite{Abrams-finite}). 

\begin{corollary}\label{cor:acyclic}
Let $G$ be an acyclic directed graph with only finitely many vertices, and let $n$ be the number of sources in $G^0$. For each category $\category{C}$ with nonempty coproducts we have an equivalence $\CKL_{\category{C}}(G) \cong \category{C}^n$.
If in addition $G$ has finitely many edges, then the same holds for all categories $\category{C}$ with binary coproducts.
\end{corollary}

(A directed graph is called \emph{acyclic} if it has no directed cycles, the latter meaning a nonempty sequence of edges $(e_1,e_2,\ldots, e_n)$ with $\target(e_i)=\source(e_{i+1})$ for $i=1,\ldots,n-1$, and $\target(e_n)=\source(e_1)$.)

\begin{proof}
If $G$ has no edges then we are in the setting of Example \ref{ex:no-edges}. If $G$ has an edge then, since $G$ has only finitely many vertices and no cycles, $G$ has at least one sink $w$ that is not a source. Using Theorem \ref{thm:sinks} to remove $w$ we obtain a graph with fewer vertices whose $\CKL_{\category{C}}$ category is equivalent to that of $G$, and now the corollary follows by induction.
\end{proof}

Corollary \ref{cor:acyclic} leads to a simple description of  $\CKL_{\category{P}}(G)$ when the directed graph $G$ has only finitely many vertices and when $\category{P}$ is the category associated to a partially ordered set $(P,\leq$). A subset $X$ of $G^0$ is called \emph{irreducible} if for each ordered pair of distinct vertices $(v,w)\in X^2$ there is a directed path in $G$ with source $v$ and target $w$ (that is, there is a sequence of edges $(e_1,\ldots,e_n)$ with $\source(e_1)=v$, $\target(e_i)=\source(e_{i+1})$ for $i=1,\ldots,n-1$, and $\target(e_n)=w$). A subset $X$ of $G^0$ is called \emph{cohereditary} if for each $e\in G^1$ with $\target(e)\in X$ we have $\source(e)\in X$.

\begin{corollary}\label{cor:P-finite}
If $G$ is a directed graph with finitely many vertices, and $\category{P}$ is the category associated to a partially ordered set $(P,\leq)$ in which every finite subset has a supremum, then $\CKL_{\category{P}}(G)$ is equivalent to $\category{P}^m$, where $m$ is the number of cohereditary irreducible subsets of $G^0$.
\end{corollary}

\begin{proof}
Since the union of two irreducible sets sharing a common vertex is again irreducible, each vertex $v\in G^0$ is contained in a unique maximal irreducible set $X_v\subseteq G^0$. Let $\widetilde{G}$ be the directed graph with vertex set $\{ X_v\ |\ v\in G^0\}$, such that for each pair of vertices $v,w\in G^0$ there is a single edge in $\widetilde{G}$ with source $X_v$ and target $X_w$ if and only if $X_v\neq X_w$ and there exists an edge $e\in G^1$ with $\source(e)\in X_v$ and $\target(e)\in X_w$. The fact that each $X_v$ is maximal irreducible ensures that the graph $\widetilde{G}$ is acyclic, and the sources in $\widetilde{G}$ are those $X_v$ that are cohereditary in $G$. Noting that a cohereditary irreducible subset of $G^0$ is automatically maximal irreducible, and that the graph $\widetilde{G}$ has only finitely many edges, we conclude from Corollary \ref{cor:acyclic} that $\CKL_{\category{P}}(\widetilde{G})\cong \category{P}^m$, where $m$ is the number of cohereditary irreducible subsets of $G^0$.

It now remains to show that $\CKL_{\category{P}}(G)\cong \CKL_{\category{P}}(\widetilde{G})$. We construct a functor $\Phi: \CKL_{\category{P}}(G)\to \CKL_{\category{P}}(\widetilde{G})$ by defining, for each diagram $\functor{D}$ in $\CKL_{\category{P}}(G)$ and each vertex $X_v\in \widetilde{G}_0$,  $\Phi(\functor{D})_{X_v}\coloneqq \functor{D}_v$. This is well defined because if $X_w= X_v$ then there are directed paths in $G$ from $v$ to $w$ and from $w$ to $v$, and so the fact that $\functor{D}$ is a diagram in $\category{P}$ ensures that $\functor{D}_v\leq \functor{D}_w \leq \functor{D}_v$. Since $\category{P}$ has at most one morphism between any pair of objects we do not need to define $\Phi(\functor{D})$ on the edges of $\widetilde{G}$, but instead we must check that if there is an edge in $\widetilde{G}$ from $X_w$ to $X_v$ then $\Phi(\functor{D})_{X_w} \leq \Phi(\functor{D})_{X_v}$. This is the case because there is a directed path in $G$ from $w$ to $v$, and so we have $\functor{D}_w\leq \functor{D}_v$. 

To check that $\Phi(\functor{D})$ satisfies the coproduct condition, fix a vertex $v\in G^0$ such that $X_v$ is not a source in $\widetilde{G}$; this ensures in particular that $v$ is not a source in $G$. The coproduct condition at $X_v$ is the requirement that 
\begin{equation}\label{eq:P-cor}
\functor{D}_v = \sup\{ \functor{D}_{\source(e)}\ |\ e\in G^1,\ \target(e)\in X_v\}.
\end{equation} 
We observed above that $\functor{D}_v \geq \functor{D}_{\source(e)}$ whenever $\target(e)\in X_v$, so $\functor{D}_v$ is an upper bound for the set on the right in \eqref{eq:P-cor}. On the other hand, the coproduct condition on $\functor{D}$ ensures that $\functor{D}_v$ is the supremum of the set $\{\functor{D}_{\source(e)}\ |\ e\in G^1,\ \target(e)=v\}$, which is a subset of the right-hand side of \eqref{eq:P-cor}, and so \eqref{eq:P-cor} holds.

The map $\Phi$ is a full and faithful functor $\CKL_{\category{P}}(G)\to \CKL_{\category{P}}(\widetilde{G})$ (in other words, $\Phi$ is an injective mapping of partially ordered sets), because for diagrams $\functor{D}$ and $\functor{E}$ in $\CKL_{\category{P}}(G)$ the definition of $\Phi$ implies immediately that for each $v\in G^0$ we have $\functor{D}_v\leq \functor{E}_v$ if and only if $\Phi(\functor{D})_{X_v}\leq \Phi(\functor{E})_{X_v}$. 

Finally, to show that $\Phi$ is essentially surjective, let $\functor{E}$ be a diagram in $\CKL_{\category{P}}(\widetilde{G})$, and let $\functor{D}$ be the diagram in $\category{P}$ of shape $G$ given by $\functor{D}_v\coloneqq \functor{E}_{X_v}$ for each $v\in G^0$. This does define a diagram because for each edge $e\in G^1$ there is an edge in $\widetilde{G}$ with source $X_{\source(e)}$ and target $X_{\target(e)}$, and consequently $\functor{D}_{\source(e)}=\functor{E}_{\source(e)} \leq \functor{E}_{\target(e)} =\functor{D}_{\target(e)}$.  To show that $\functor{D}$ satisfies the coproduct condition, we note that this condition requires that for each $v\in G^0$ we have
\begin{equation}\label{eq:P-cor2}
\functor{E}_{X_v} = \sup\{ \functor{E}_{X_{\source(e)}}\ |\ e\in G^1,\ \target(e)=v\}.
\end{equation}
The fact that $\functor{E}$ satisfies the coproduct condition ensures that
\begin{equation}\label{eq:P-cor3}
\functor{E}_{X_v} = \sup\{ \functor{E}_{X_{\source(f)}}\ |\ f\in G^1,\ \target(f)\in X_v\},
\end{equation}
and since the set on the right-hand side of \eqref{eq:P-cor3} contains the set on the right-hand side of \eqref{eq:P-cor2} we see that $\functor{E}_{X_v}$ is an upper bound for the set on the right-hand side in \eqref{eq:P-cor2}. On the other hand, for each edge $f\in G^1$ with $\target(f)\in X_v$ there is an edge $e\in G^1$ with $\target(e)=v$ such that $\functor{E}_{X_{\source(f)}}\leq \functor{E}_{X_{\source(e)}}$: indeed if $\target(f)\in X_v$ then $f$ is the first edge in a directed path in $G$ with source $\source(f)$ and target $v$, and we can take $e$ to be the final edge in such a path. Thus every element of the set on the right-hand side of \eqref{eq:P-cor3} is dominated by some element of the set on the right-hand side of \eqref{eq:P-cor2}, and we conclude that both sets have the same supremum, $\functor{E}_{X_v}$. Thus $\functor{D}$ satisfies the coproduct condition, and this completes the proof.
\end{proof}

\begin{example}\label{ex:Cuntz-P}
If $G$ is a finite irreducible graph (meaning that the entire vertex-set $G^0$ is irreducible in the sense defined above) then we have $\CKL_{\category{P}}(G)\cong \category{P}$ for every category $\category{P}$ associated to a partially ordered set. (No hypothesis need be made here on the existence of suprema in $\category{P}$, because the graph $\widetilde{G}$ appearing in the proof of Corollary \ref{cor:P-finite} consists in this case of a single vertex and no edges.) For example, the graphs $G$ and $H$ from \eqref{eq:intro-Cuntz} satisfy $\CKL_{\category{P}}(G)\cong \CKL_{\category{P}}(H)$ for every partially ordered set $\category{P}$, and this gives one answer to part (3) of Question \ref{q:Cuntz}.
\end{example}

\begin{remark}\label{remark:P-infinite}
For graphs $G$ with infinitely many vertices it is not necessarily true that $\CKL_{\category{P}}(G)$ is equivalent to a product of copies of $\category{P}$, even if we assume $\category{P}$ to have all coproducts. For example, let $\category{P}$ be the category associated to a partially ordered set $(P,\leq)$, and let $G$ be the graph 
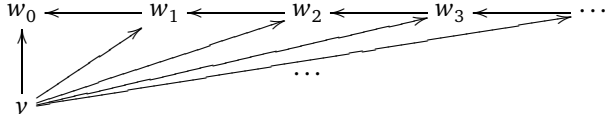
\[
\xygraph{
[u(.5)]{w_0}="y"([d]{v}="x",[r(1.5)]{w_1}="b1" [r(1.5)]{w_2}="b2" [r(1.5)]{w_3}="b3" [r(1.5)]{\cdots}="dots"), "x":"y", "x":"b1":"y", "x":"b2":"b1", "x":"b3":"b2", "x":_-*{\cdots}"dots":"b3"
}
\]
where the path of $w$s extends infinitely to the right, each $w_i$ being the target of one edge with source $v$. For each diagram $\functor{D}$ in $\CKL_{\category{P}}(G)$ and each $i\geq 0$ we have $\functor{D}_{w_i}=\sup\{ \functor{D}_v, \functor{D}_{w_{i+1}}\}$ and $\functor{D}_v\leq \functor{D}_{w_{i+1}}$, giving $\functor{D}_{w_0}=\functor{D}_{w_i}$ for all $i$. Thus the map $\functor{D}\mapsto (\functor{D}_v ,\functor{D}_{w_0})$ gives an equivalence $\CKL_{\category{P}}(G)\to \category{Arr}(\category{P})$, the arrow category of $\category{P}$ (i.e., the category associated to the subset $\{ (x,y) \ |\ x\leq y\}$ of the partially ordered set $P^2$). It is easy to come up with examples of $\category{P}$s for which $\category{Arr}(\category{P})$ is not equivalent to any power of $\category{P}$: for example, take $P=\{0,1\}$ with $0\leq 1$.
\end{remark}

\subsection{Out-delays}

\begin{definition}\label{def:out-delay}
Let $G=(G^0,G^1,\source,\target)$ be a directed graph, and let $d:G^0\sqcup G^1\to \N\sqcup\{\infty\}$ be a function such that for each edge $e\in G^1$ we have $d(e)<\infty$ and $d(e)\leq d(\source(e))$. (In this paper $\N=\{0,1,2,\ldots\}$, and $n< \infty$ for every $n\in \N$.) We define a new graph $G_{\mathrm{od},d}=(G_{\mathrm{od},d}^0, G_{\mathrm{od},d}^1, \source_{\mathrm{od},d}, \target_{\mathrm{od},d})$, called the \emph{out-delay} of $G$ determined by $d$, as follows:
\[
\begin{aligned}
& G^0_{\mathrm{od},d} = \{ (v,n)\ |\ v\in G^0, n\in \N,\ 0\leq n\leq d(v)\} \\
& G^1_{\mathrm{od},d} = G^1 \sqcup \{ e_{v,n}\ |\ v\in G^0, n\in \N,\ 1\leq n\leq d(v)\} \\
& \source_{\mathrm{od},d}(e) = (\source(e), d(e)), \quad \source_{\mathrm{od},d}(e_{v,n})=(v,n-1)
\\
& \target_{\mathrm{od},d}(e) = (\target(e),0), \qquad \ \,\target_{\mathrm{od},d}(e_{v,n})=(v,n).
\end{aligned}
\]
\end{definition}

\begin{example}\label{ex:out-delay}
If $G$ is the directed graph
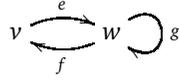
\[
\xygraph{
{v}="v":^-{e}@/^/[r] {w}="w":^-{f}@/^/"v","w":^{g}@(ur,dr)"w"
}
\]
and $d:G^0\sqcup G^1 \to \N\sqcup\{\infty\}$ is the function
\[
d(v)=d(e)=0,\quad d(w)=\infty,\quad d(f)=1,\quad d(g)=2,
\]
then $G_{\mathrm{od},d}$ is the directed graph
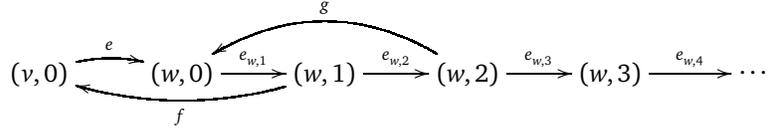
\[
\xygraph{
{(v,0)}="v0":^-{e}@/^/[r(1.5)]{(w,0)}="w0":^-{e_{w,1}}[r(1.5)]{(w,1)}="w1"
(:^-{f}@/^10pt/"v0",:^-{e_{w,2}}[r(1.5)]{(w,2)}="w2"(:_-{g}@/_20pt/"w0",:^-{e_{w,3}}[r(1.5)]{(w,3)} :^-{e_{w,4}}[r(1.5)]{\cdots}
)
)
}
\]
where the path of $(w,n)$s continues infinitely to the right.
\end{example}

\begin{example}\label{ex:tails}
Let $G=(G^0,G^1,\source,\target)$ be a directed graph, and let $d:G^0\sqcup G^1 \to \N\sqcup\{\infty\}$ be the function defined for $v\in G^0$ and $e\in G^1$ by
\[
d(v) = \begin{cases} 0 & \text{if $v$ is not a sink}\\ \infty & \text{if $v$ is a sink}\end{cases}
\qquad \text{and}\qquad d(e)=0.
\]
The graph $G_{\mathrm{od},d}$ is obtained from $G$ by ``adding a tail'' at each sink; for example:
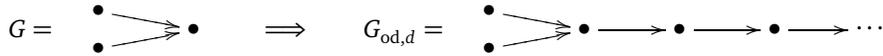
\[
G= \quad \xygraph{
{\closed}="sink"([l][u(.2)]{\closed}="u":"sink",[l][d(0.2)]{\closed}="b":"sink")
}
\qquad
\Longrightarrow
\qquad
G_{\mathrm{od},d}= \quad \xygraph{
{\closed}="sink"([l][u(.2)]{\closed}="u":"sink",[l][d(0.2)]{\closed}="b":"sink",:[r]{\closed}:[r]{\closed}:[r]{\cdots})
}
\]
\end{example}

\begin{theorem}\label{thm:out-delay}
Let $G_{\mathrm{od},d}$ be an out-delay of a directed graph $G$, as in Definition \ref{def:out-delay}. For each category $\category{C}$ there is an equivalence of categories 
\(
\CKL_{\category{C}}(G) \cong \CKL_{\category{C}}(G_{\mathrm{od},d}).
\)
\end{theorem}

\begin{proof}
To simplify the notation we will omit the ``$\mathrm{od}$'' and just write $G_d=(G^0_d,G^1_d,\source_d,\target_d)$. 
For each diagram $\functor{D}$ in $\CKL_{\category{C}}(G)$ let $\Phi (\functor{D})$ be the diagram in $\category{C}$ of shape $G_{d}$ defined by
\[
\Phi (\functor{D})_{(v,n)} = \functor{D}_v,\quad \Phi (\functor{D})_e = \functor{D}_e, \quad \Phi (\functor{D})_{e_{v,n}} = \id_{\functor{D}_v}.
\]
For the graph $G$ and the function $d$ in Example \ref{ex:out-delay}, for instance, $\Phi$ sends each diagram
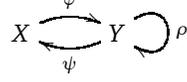
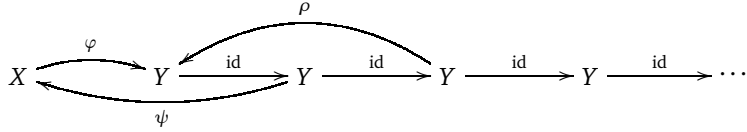
\[
\xygraph{
{X}="v":^-{\phi}@/^/[r] {Y}="w":^-{\psi}@/^/"v","w":^{\rho}@(ur,dr)"w"
}
\]
to the diagram
\[
\xygraph{
{X}="v0":^-{\phi}@/^/[r(1.5)]{Y}="w0":^-{\id}[r(1.5)]{Y}="w1"
(:^-{\psi}@/^10pt/"v0",:^-{\id}[r(1.5)]{Y}="w2"(:_-{\rho}@/_20pt/"w0",:^-{\id}[r(1.5)]{Y} :^-{\id}[r(1.5)]{\cdots}
)
)
}
\]

The diagram $\Phi (\functor{D})$ satisfies the coproduct condition: for the vertices $(v,0)$ this follows immediately from the fact that $\functor{D}$ satisfies the coproduct condition, while for the vertices $(v,n)$ with $n\geq 1$ this follows from the fact that $(v,n)\in G^0_{d}$ is the target of exactly one edge $e_{(v,n)}\in G^1_{d}$, where $\Phi (\functor{D})_{e_{(v,n)}}$ is an isomorphism.

We extend $\Phi $ to a functor $\CKL_{\category{C}}(G)\to \CKL_{\category{C}}(G_{d})$ by defining, for each morphism $T:\functor{D}\to \functor{E}$ in $\CKL_{\category{C}}(G)$, a morphism $\Phi (T): \Phi (\functor{D})\to \Phi (\functor{E})$ in $\CKL_{\category{C}}(G_d)$ with $\Phi (T)_{(v,n)}\coloneqq T_v$ for each $(v,n)\in G^0_{d}$. 

Since each vertex $v\in G^0$ gives rise to at least one vertex $(v,n)\in G^0_{d}$, each component $T_v$ of $T$ appears as one of the components $\Phi (T)_{(v,n)}$ of $\Phi (T)$, and so the functor $\Phi $ is faithful.

To see that $\Phi $ is full, let $\functor{D}$ and $\functor{E}$ be diagrams in $\CKL_{\category{C}}(G)$, and let $S:\Phi (\functor{D})\to \Phi (\functor{E})$ be a morphism in $\CKL_{\category{C}}(G_{d})$. Notice that for each vertex $(v,n)\in G^0_{d}$ with $n\geq 1$ we have
\[
S_{(v,n)} =  S_{\target_d(v,n)}\circ \Phi (\functor{D})_{e_{v,n}} = \Phi (\functor{E})_{e_{v,n}} \circ S_{\source_d(v,n)} = S_{(v,n-1)},
\]
and so $S_{(v,n)}=S_{(v,0)}$ for all $n$. Now for each $v\in G$ define a morphism $T_v : \functor{D}_v\to \functor{E}_v$ in $\category{C}$ by $T_v \coloneqq S_{(v,0)}$. For each edge $e\in G^1$ we have
\[
T_{\target(e)}\circ \functor{D}_e = S_{\target_d(e)}\circ \Phi (\functor{D})_e = \Phi (\functor{E})_e\circ S_{\source_d(e)} = \functor{E}_e\circ T_{\source(e)},
\]
and this shows that the collection of morphisms $T = (T_v\ |\ v\in G^0)$ is a morphism in $\CKL_{\category{C}}(G)$. For each vertex $(v,n)\in G^0_d$ we have 
\[
S_{(v,n)}=S_{(v,0)} = T_v =\Phi (T)_{(v,n)}
\]
and so $\Phi $ is full.

Finally, to show that the functor $\Phi $ is essentially surjective we let $\functor{E}$ be a diagram in $\CKL_{\category{C}}(G_d)$. Let $\functor{D}$ be the diagram of shape $G$ in $\category{C}$ given on $v\in G^0$ and on $e\in G^1$ by
\[
\functor{D}_v \coloneqq \functor{E}_{(v,0)} \qquad \text{and}\qquad \functor{D}_e \coloneqq \begin{cases} \functor{E}_e & \text{if $d(e)=0$,} \\ \functor{E}_e\circ \functor{E}_{e_{\source(e),d(e)}}\circ \functor{E}_{e_{\source(e),d(e)-1}} \circ \cdots \circ \functor{E}_{e_{\source(e),1}} & \text{if $d(e)\geq 1$.}\end{cases}
\]
The fact that this $\functor{D}$ satisfies the coproduct condition follows immediately from the fact that $\functor{E}$ does so, using in particular the fact that each $\functor{E}_{e_{v,n}}$ is an isomorphism. We will prove that $\Phi (\functor{D})\cong \functor{E}$ by considering the maps $T_{(v,n)} : \Phi (\functor{D})_{(v,n)}= \functor{E}_{(v,0)} \to \functor{E}_{(v,n)}$ defined by
\[
T_{(v,n)} \coloneqq \begin{cases} \id_{\functor{E}_{(v,0)}} & \text{if $n=0$,} \\ \functor{E}_{e_{v,n}}\circ \functor{E}_{e_{v,n-1}} \circ \cdots \circ \functor{E}_{e_{v,1}} & \text{if $n\geq 1$}. \end{cases}
\]
Each $T_{(v,n)}$ is an isomorphism in $\category{C}$, because each $\functor{E}_{e_{v,m}}$ is an isomorphism. The family of morphisms $T=(T_{(v,n)}\ |\ (v,n)\in G_d^0)$ is a morphism in $\CKL_{\category{C}}(G_d)$ from $\Phi (\functor{D})$ to $\Phi (\functor{E})$ because for each edge $e\in G^1$ we have
\[
T_{\target_d(e)} \circ \Phi (\functor{D})_e 
 = \id_{{\functor{E}_{(\target(e),0)}}} \circ \functor{E}_e\circ \functor{E}_{e_{\source(e),d(e)}}\circ \functor{E}_{e_{\source(e),d(e)-1}} \circ \cdots \circ \functor{E}_{e_{\source(e),1}}
 = \functor{E}_e \circ T_{\source_d(e)},
\]
while for each $e_{v,n}\in G_d^1$ we have
\[
T_{\target_d(e_{v,n})} \circ \Phi (\functor{D})_{e_{v,n}} = T_{(v,n)} \circ \id_{\functor{E}_{(v,0)}} = \functor{E}_{e_{v,n}}\circ \functor{E}_{e_{v,n-1}} \circ \cdots \circ \functor{E}_{e_{v,1}} 
= \functor{E}_{e_{v,n}} \circ T_{\source_d(e_{v,n})}.
\]
Thus $T: \Phi (\functor{D}) \to \functor{E}$ is an isomorphism in $\CKL_{\category{C}}(G_d)$, and so $\Phi $ is essentially surjective.
\end{proof}

\subsection{Adding heads at sources}

\begin{definition}\label{def:heads}
Let $G=(G^0,G^1,\source, \target)$ be a directed graph, and let $G^0_0$ be the set of sources in $G$. Define a new graph $\widehat{G}=(\widehat{G}^0,\widehat{G}^1,\hat{\source},\hat{\target})$ by
\[
\begin{aligned}
& \widehat{G}^0 = G^0 \sqcup \{ (v,n)\ |\ v\in G^0_0,\ n=1,2,3,\ldots\} \\
& \widehat{G}^1 = G^1 \sqcup \{ e_{v,n}\ |\ v\in G^0_0,\ n=1,2,3,\ldots\} \\
& \hat{\source}(e)=\source(e), \qquad \hat{\source}(e_{v,n}) = (v,n),\\
& \hat{\target}(e)=\target(e), \qquad \hat{\target}(e_{v,n}) = \begin{cases} (v,n-1) & \text{if $n\geq 2$}\\ v & \text{if $n=1$.}
\end{cases}
\end{aligned}
\]
\end{definition}

For example, if $G$ is the graph
\begin{equation}\label{eq:heads-ex1}
\xygraph{
{v}(:@/^/[r]{\closed}="w",:@/_/"w":@(ur,dr)"w")
}
\end{equation}
then $\widehat{G}$ is the graph 
\[
\xygraph{
{\cdots}:^-{e_{v,3}}[r(1.5)]{(v,2)}:^-{e_{v,2}}[r(1.5)]{(v,1)}:^-{e_{v,1}}[r(1.5)]{v}(:@/^/[r]{\closed}="w",:@/_/"w":@(ur,dr)"w")
}.
\]
Notice that in general $\widehat{G}$ has no sources, and that if $G$ has no sources then $\widehat{G}=G$.

\begin{theorem}\label{thm:heads}
Let $G$ be a directed graph, and let $\widehat{G}$ be the graph obtained by adding a head at each source of $G$ as in Definition \ref{def:heads}. For each category $\category{C}$ there is an equivalence of categories  
\(
\CKL_{\category{C}}(G) \cong \CKL_{\category{C}}( \widehat{G}).
\)
\end{theorem}

\begin{proof}
For each diagram $\functor{D}$ in $\CKL_{\category{C}}(G)$ we define a diagram $\Phi (\functor{D})$ of shape $\widehat{G}$ by setting
\[
\Phi (\functor{D})_v = \functor{D}_v,\quad \Phi (\functor{D})_e = \functor{D}_e, \quad \Phi (\functor{D})_{(v,n)} \coloneqq \functor{D}_v,\quad \text{and}\quad \Phi (\functor{D})_{e_{v,n}} \coloneqq \id_{\functor{D}_v}.
\]
For the graph \eqref{eq:heads-ex1}, for instance, $\Phi$ sends each diagram
\[
\xygraph{
{X}(:^-{\phi}@/^/[r]{Y}="w",:_-{\psi}@/_/"w":^-{\rho}@(ur,dr)"w")
}
\]
to the diagram
\[
\xygraph{
{\cdots}:^-{\id}[r]{X}:^-{\id}[r]{X}:^-{\id}[r]{X}(:^-{\phi}@/^/[r]{Y}="w",:_-{\psi}@/_/"w":^-{\rho}@(ur,dr)"w")
}.
\]

The diagram $\Phi(\functor{D})$ clearly inherits the coproduct condition from the diagram $\functor{D}$, and so $\Phi (\functor{D})$ is an object of $\CKL_{\category{C}}(\widehat{G})$. 

If $T:\functor{D}\to \functor{E}$ is a morphism in $\CKL_{\category{C}}(G)$ then we let $\Phi (T) : \Phi (\functor{D})\to \Phi (\functor{E})$ be the morphism in $\CKL_{\category{C}}(\widehat{G})$ define by
\[
\Phi (T)_v \coloneqq T_v \qquad \text{and}\qquad \Phi (T)_{(v,n)} \coloneqq T_v.
\]
In this way we obtain a functor $\Phi  : \CKL_{\category{C}}(G)\to \CKL_{\category{C}}(\widehat{G})$. 

If $T:\functor{D}\to \functor{E}$ is a morphism in $\CKL_{\category{C}}(G)$, then each component $T_v$ of $T$ appears as a component of the morphism $\Phi (T)$, and it follows from this that $\Phi $ is faithful. To see that $\Phi $ is full, let $S:\Phi (\functor{D})\to \Phi (\functor{E})$ be a morphism in $\CKL_{\category{C}}(\widehat{G})$, and note that for each source $v$ of $G$ and each $n\geq 1$ we have 
\[
S_{(v,n)} = \Phi (\functor{E})_{e_{v,n}}  \circ S_{(v,n)}  = S_{(v,n-1)}\circ \Phi (\functor{D})_{e_{v,n}} = S_{(v,n-1)}
\]
where $(v,0)\coloneqq v$. Thus $S_{(v,n)}=S_v$ for each source $v$ and each $n\geq 1$, and so we have $S = \Phi (T)$ where $T:\functor{D}\to \functor{E}$ is the morphism defined by $T_v\coloneqq S_v$ for each $v\in G^0$.

Finally, to show that $\Phi $ is essentially surjective, let $\functor{E}$ be a diagram in $\CKL_{\category{C}}(\widehat{G})$. Note that for each source $v$ of $G$ and each $n\geq 1$ the morphism $\functor{E}_{e_{v,n}}:\functor{E}_{(v,n)}\to \functor{E}_{(v,n-1)}$ is an isomorphism, because $(v,n-1)$ is the target of only one edge (cf Example \ref{ex:one-edge}). Let $\functor{D}$ be the object in $\CKL_{\category{C}}(G)$ defined by $\functor{D}_v\coloneqq \functor{E}_v$ and $\functor{D}_e\coloneqq \functor{E}_e$ for each $v\in G^0$ and each $e\in G^1$. The diagram $\functor{D}$ satisfies the coproduct condition because $\functor{E}$ does, and the maps
\[
S_v : \functor{E}_v \xrightarrow{\id} \functor{E}_v \qquad \text{and}\qquad S_{(v,n)} : \functor{E}_v \xrightarrow{(\functor{E}_{e_{v,1}}\circ \functor{E}_{e_{v,2}}\circ \cdots \circ \functor{E}_{e_{v,n}})^{-1}} \functor{E}_{(v,n)}
\]
assemble to give an isomorphism $S:\Phi (\functor{D})\to \functor{E}$ in $\CKL_{\category{C}}(\widehat{G})$. Thus $\Phi $ is essentially surjective.
\end{proof}

\subsection{In-delays}

\begin{definition}\label{def:in-delay}
Let $G=(G^0,G^1,\source,\target)$ be a directed graph with no infinite receivers, and let $d:G^0\sqcup G^1 \to \N$ be a function such that for each vertex $v\in G^0$ we have 
\[
d(v)=\begin{cases} \sup_{e\in \target^{-1}(v)} d(e) & \text{if $v$ is not a source,}\\
0 & \text{if $v$ is a source}
\end{cases}
\] 
We define a new graph $G_{\mathrm{id},d}=(G^0_{\mathrm{id},d},G^1_{\mathrm{id},d}, \source_{\mathrm{id},d}, \target_{\mathrm{id},d})$, called the \emph{in-delay} of $G$  determined by the function $d$, as follows:
\[
\begin{aligned}
& G^0_{\mathrm{id},d} = \{ (v,n)\ |\ v\in G^0, n\in \N,\ 0\leq n\leq d(v)\} \\
& G^1_{\mathrm{id},d} = G^1 \sqcup \{ e_{v,n}\ |\ v\in G^0, n\in \N,\ 1\leq n\leq d(v)\} \\
& \source_{\mathrm{id},d}(e) = (\source(e), 0), \qquad\ \! \source_{\mathrm{id},d}(e_{v,n})=(v,n)
\\
& \target_{\mathrm{id},d}(e) = (\target(e),d(e)), \quad\target_{\mathrm{id},d}(e_{v,n})=(v,n-1).
\end{aligned}
\]
\end{definition}

For example, if $G$ is the directed graph
\begin{equation}\label{eq:in-delay-ex1}
\xygraph{
{v}="v":^-{e}@/^/[r] {w}="w":^-{f}@/^/"v","w":^{g}@(ur,dr)"w"
}
\end{equation}
and $d:G^0\sqcup G^1 \to \N$ is the function
\[
d(v)=d(f)=0, \quad d(w)=d(g)=2,\quad d(e)=1
\]
then $G_{\mathrm{id},d}$ is the graph
\begin{equation}\label{eq:in-delay-ex2}
\xygraph{
{(w,2)}="w2":_-{e_{w,2}}[l(1.5)] {(w,1)}="w1":_-{e_{w,1}}[l(1.5)] {(w,0)}="w0"(:_-{g}@/_15pt/"w2", :^-{f}@/^15pt/[l(1.5)]{(v,0)}="v0":^-{e}@/^15pt/"w1")
}
\end{equation}

\begin{theorem}\label{thm:in-delay}
Let $G$ be a directed graph with no infinite receivers, and let $G_{\mathrm{id},d}$ be an in-delay of $G$ as in Definition \ref{def:in-delay}. For each category $\category{C}$ with binary coproducts there is an equivalence of categories 
\(
\CKL_{\category{C}}(G) {\cong} \CKL_{\category{C}}(G_{\mathrm{id},d}).
\)
\end{theorem}

\begin{proof}
To simplify the notation we will omit the ``id'' and just write $G_d=(G_d^0, G_1^0, \source_d, \target_d)$. We may assume without loss of generality that $G$ has no sources. To see why, consider the graph $\widehat{G}$ obtained by adding heads at the sources of $G$ as in Definition \ref{def:heads}, and extend the function $d$ to $\widehat{G}^0\sqcup \widehat{G}^1$ by setting $d(v,n)=d(e_{v,n})=0$ for each source $v$ of $G$ and each $n\geq 1$. We then have $(\widehat{G})_d = \widehat{(G_d)}$, and so if we prove Theorem \ref{thm:in-delay} for graphs (like $\widehat{G}$) with no sources, then that result combined with Theorem \ref{thm:heads} will give a chain of equivalences
\[
\CKL_{\category{C}}(G) \xrightarrow{\cong} \CKL_{\category{C}}(\widehat{G}) \xrightarrow{\cong} \CKL_{\category{C}}( (\widehat{G})_d) = \CKL_{\category{C}}\left(\widehat{(G_d)}\right) \xrightarrow{\cong} \CKL_{\category{C}}(G_d)
\]
as required.

So we assume for the rest of the proof that $G$ has no sources, and we note that the same is then true of $G_d$. For each diagram $\functor{D}$ in $\CKL_{\category{C}}(G)$ let $\Phi (\functor{D})$ be the diagram in $\category{C}$ of shape $G_d$ defined as follows. For each vertex $v\in G^0$ and each $n\leq d(v)$ we define
\[
\Phi (\functor{D})_{(v,n)} \coloneqq 
\bigsqcup_{\substack{e\in \target^{-1}(v),\\ d(e)\geq n}} \functor{D}_{\source(e)}. 
\]
Note that the assumptions that $G$ has no sources or infinite receivers, and that $d(v)=\sup_{e\in \target^{-1}(v)} d(e)$, together ensure that the coproduct on the right-hand side above is indexed by a nonempty finite set, and thus exists (uniquely, up to canonical isomorphism) in the category $\category{C}$. For each edge $e\in G^1$ we let
\begin{equation}\label{eq:in-delay-PhiMe}
\Phi (\functor{D})_e \coloneqq \bigsqcup_{f\in \target^{-1}\source(e)} \functor{D}_f : \bigsqcup_{f\in \target^{-1}\source(e)} \functor{D}_{\source(f)} \to \bigsqcup_{\substack{g\in \target^{-1}\target(e) \\ d(g)\geq d(e)}} \functor{D}_{\source(g)}.
\end{equation}
To make sense of this definition, note that for each $f\in \target^{-1}\source (e)$ we have $\functor{D}_f : \functor{D}_{\source(f)}\to \functor{D}_{\source(e)}$, and $e$ is one of the edges $g$ indexing the coproduct on the right-hand side. So we can define the coproduct of the maps $\functor{D}_f$ as explained in Section \ref{subsec:coproducts}.

Finally, for each edge of the form $e_{v,n}$ in $G_d^1$, where $v\in G^0$ and $1\leq n\leq d(v)$, we define
\begin{equation}\label{eq:in-delay-evn-def}
\Phi (\functor{D})_{e_{v,n}} \coloneqq \bigsqcup_{\substack{e\in \target^{-1}(v),\\ d(e)\geq n}} \id_{\functor{D}_{\source(e)}}: \bigsqcup_{\substack{e\in \target^{-1}(v),\\ d(e)\geq n}} \functor{D}_{\source(e)} \to \bigsqcup _{\substack{f\in \target^{-1}(v),\\ d(f)\geq n-1}} \functor{D}_{\source(f)}.
\end{equation}
To make sense of this definition, note that the set of indices for the coproduct in the domain is a subset of the set of indices for the coproduct in the codomain.

To illustrate: for the graphs $G$ and $G_{d}$ in \eqref{eq:in-delay-ex1} and \eqref{eq:in-delay-ex2}, our map $\Phi$ sends each diagram
\[
\xygraph{
{X}="v":^-{\phi}@/^/[r] {Y}="w":^-{\psi}@/^/"v","w":^{\rho}@(ur,dr)"w"
}
\]
of shape $G$ to the diagram
\[
\xygraph{
{Y}="w2":_-{\text{can}}[l(1.5)] {X\sqcup Y}="w1":_-{\id}[l(1.5)] {X\sqcup Y}="w0"(:_-{\phi\sqcup\rho}@/_15pt/"w2", :^-{\phi\sqcup\rho}@/^15pt/[l(1.5)]{Y}="v0":^-{\psi}@/^15pt/"w1")
}
\]
of shape $G_{d}$. (Here `$\text{can}$' denotes the canonical morphism into the coproduct.)

Returning to the general case, to check that the diagram $\Phi (\functor{D})$ satisfies the coproduct condition we let $(v,n)$ be a vertex of $G_d$, and suppose first that $n<d(v)$. We can then write
\[
\Phi (\functor{D})_{(v,n)} = \left(\bigsqcup_{\substack{e\in \target^{-1}(v)\\ d(e)=n}} \functor{D}_{\source(e)}\right) \sqcup \left( \bigsqcup_{\substack{e\in \target^{-1}(v)\\ d(e) \geq n+1}} \functor{D}_{\source(e)}\right).
\]
Now $\target_d^{-1}(v,n) = \{ e\in \target^{-1}(v)\ |\ d(e)=n\} \sqcup \{e_{v,n+1}\}$, and consulting the definition of $\Phi (\functor{D})$ shows that the map $\bigsqcup_{e\in \target_d^{-1}(v,n)}\Phi (\functor{D})_{e} : \bigsqcup_e \Phi (\functor{D})_{\source_d(e)} \to \Phi (\functor{D})_v$ is the coproduct of the vertical arrows in the diagram
\[
\xymatrix@C=5pt{
\displaystyle \bigsqcup_{\substack{e\in \target^{-1}(v)\\ d(e)=n}} \bigsqcup_{f\in \target^{-1}\source(e)} \functor{D}_{\source(f)} \ar[d]_-{\bigsqcup_e \bigsqcup_f \functor{D}_f} & \sqcup & \displaystyle\bigsqcup_{\substack{e\in \target^{-1}(v)\\ d(e)\geq n+1}} \functor{D}_{\source(e)} \ar[d]^-{\bigsqcup_e \id} 
\\
\displaystyle \bigsqcup_{\substack{e\in \target^{-1}(v)\\ d(e)=n}} \functor{D}_{\source(e)} & \sqcup &
\displaystyle \bigsqcup_{\substack{e\in \target^{-1}(v)\\ d(e) \geq n+1}} \functor{D}_{\source(e)}
}.
\]
The coproduct property of $\functor{D}$ ensures that this map is an isomorphism. The verification of the coproduct condition at each vertex $(v,n)$ with $n=d(v)$ is similar to the above, except that the $d(e)=n+1$ terms are absent. We conclude that $\Phi (\functor{D})$ does satisfy the coproduct condition, and hence is an object of $\CKL_{\category{C}}(G_d)$. 

To describe how we turn $\Phi $ into a functor, we let $T:\functor{D}\to \functor{E}$ be a morphism in $\CKL_{\category{C}}(G)$. For each vertex $v\in G^0$, and for each $n\leq d(v)$, we define
\begin{equation}\label{eq:in-delay-PhiTvn}
\Phi (T)_{(v,n)} \coloneqq
\bigsqcup_{\substack{e\in \target^{-1}(v),\\ d(e)\geq n}} T_{\source(e)} : \bigsqcup_{\substack{e\in \target^{-1}(v),\\ d(e)\geq n}} \functor{D}_{\source(e)} \to \bigsqcup_{\substack{e\in \target^{-1}(v),\\ d(e)\geq n}} \functor{E}_{\source(e)}.
\end{equation}
The fact that $T$ is a morphism of diagrams $\functor{D}\to \functor{E}$ ensures that $\Phi (T)$ is a morphism of diagrams $\Phi (\functor{D})\to \Phi (\functor{E})$, and since $\Phi $ obviously preserves identities and commutes with composition we have defined a functor $\Phi :\CKL_{\category{C}}(G) \to \CKL_{\category{C}}(G_d)$.

To see that the functor $\Phi $ is faithful, note that for each morphism $T:\functor{D}\to \functor{E}$ in $\CKL_{\category{C}}(G)$, and each vertex $v\in G^0$ we have a diagram
\[
\xymatrix@C=60pt{
\displaystyle\bigsqcup_{e\in \target^{-1}(v)} \functor{D}_{\source(e)} \ar[d]_-{\bigsqcup_e T_{\source(e)}} \ar[r]^-{\bigsqcup_e \functor{D}_e} & \functor{D}_v \ar[d]^-{T_v} \\
\displaystyle\bigsqcup_{e\in \target^{-1}(v)} \functor{E}_{\source(e)} \ar[r]^-{\bigsqcup_e \functor{E}_e} & \functor{E}_v
}
\]
that commutes because $T$ is a morphism of diagrams. The left-hand vertical arrow is $\Phi (T)_{(v,0)}$, and the horizontal arrows are isomorphisms because $\functor{D}$ and $\functor{E}$ satisfy the coproduct condition, and so $\Phi (T)$ determines $T$.

We next prove that the functor $\Phi $ is full, which will require some work. Let $\functor{D}$ and $\functor{E}$ be diagrams in $\CKL_{\category{C}}(G)$, and let $S:\Phi (\functor{D})\to \Phi (\functor{E})$ be a morphism in $\CKL_{\category{C}}(G_d)$. For each vertex $v\in G^0$ we let $T_v : \functor{D}_v\to \functor{E}_v$ be defined by the following commutative diagram, in which the horizontal arrows are isomorphisms:
\begin{equation}\label{eq:in-delay-comm1}
\xymatrix@C=60pt{
\displaystyle\bigsqcup_{e\in \target^{-1}(v)} \functor{D}_{\source(e)} \ar[d]_-{S_{(v,0)}} \ar[r]^-{\bigsqcup_e \functor{D}_e} & \functor{D}_v \ar[d]^-{T_v} \\
\displaystyle\bigsqcup_{e\in \target^{-1}(v)} \functor{E}_{\source(e)} \ar[r]^-{\bigsqcup_e \functor{E}_e} & \functor{E}_v
}
\end{equation}
We claim that the collection of morphisms $T=(T_v\ |\ v\in G^0)$ is a morphism in $\CKL_{\category{C}}(G)$ from $\functor{D}$ to $\functor{E}$, or in other words that $\functor{E}_e \circ T_{\source(e)} = T_{\target(e)} \circ \functor{D}_e$ for every $e\in G^1$. 

To verify this equality we first note that the equality
\begin{equation*}\label{eq:in-delay-comm2}
\Phi (\functor{E})_e \circ S_{(\source(e),0)} = S_{(\target(e),d(e))}\circ \Phi (\functor{D})_{e}
\end{equation*}
(which holds because $S$ is a morphism in $\CKL_{\category{C}}(G_d)$), combined with the definition \eqref{eq:in-delay-comm1} of $T_{\source(e)}$ and the definition \eqref{eq:in-delay-PhiMe} of $\Phi (\functor{D})_{e}$, yields a commuting diagram
\[
\xymatrix@C=60pt{
\functor{D}_{\source(e)} \ar[d]_-{T_{\source(e)}} \ar[r]^-{(\bigsqcup_{f\in \target^{-1}\source(e)} \functor{D}_f)^{-1}} 
& 
\displaystyle \bigsqcup_{f\in \target^{-1}\source(e)} \functor{D}_{\source(f)} \ar[d]^-{S_{(\source(e),0)}} \ar[r]^-{\bigsqcup_{f\in \target^{-1}\source(e)}\functor{D}_f} 
&\displaystyle \bigsqcup_{\substack{g\in \target^{-1}\target(e)\\ d(g)\geq d(e)}} \functor{D}_{\source(g)} \ar[d]^-{S_{(\target(e),d(e))}} \\
\functor{E}_{\source(e)} \ar[r]^-{(\bigsqcup_{f\in \target^{-1}\source(e)} \functor{E}_f)^{-1}} 
& 
\displaystyle\bigsqcup_{f\in \target^{-1}\source(e)} \functor{E}_{\source(f)}  \ar[r]^-{\bigsqcup_{f\in \target^{-1}\source(e)}\functor{E}_f} 
& \displaystyle \bigsqcup_{\substack{g\in \target^{-1}\target(e)\\ d(g)\geq d(e)}} \functor{E}_{\source(g)}
}
\]
Upon composing the horizontal arrows in the above diagram we obtain a commuting diagram 
\begin{equation}\label{eq:in-delay-full-patch1}
\xymatrix@C=50pt{
\functor{D}_{\source(e)} \ar[d]_-{T_{\source(e)}} \ar[r]^-{\text{canonical}} &\displaystyle \bigsqcup_{\substack{g\in \target^{-1}\target(e)\\ d(g)\geq d(e)}} \functor{D}_{\source(g)} \ar[d]^-{S_{(\target(e),d(e))}} \\
\functor{E}_{\source(e)} \ar[r]^-{\text{canonical}} &\displaystyle \bigsqcup_{\substack{g\in \target^{-1}\target(e)\\ d(g)\geq d(e)}} \functor{E}_{\source(g)}
}
\end{equation}
where the horizontal arrows  are the canonical morphisms arising from the fact that $e$ is one of the indices $g$ appearing in the coproducts. 

Next we observe that the equality
\begin{equation*}
\begin{aligned}
&  \Phi (\functor{E})_{e_{\target(e),1}} \circ  \Phi (\functor{E})_{e_{\target(e),2}} \circ \cdots \circ  \Phi (\functor{E})_{e_{\target(e),d(e)}} \circ S_{(\target(e),d(e))} \\
& = S_{(\target(e),0)} \circ  \Phi (\functor{D})_{e_{\target(e),1}} \circ  \Phi (\functor{D})_{e_{\target(e),2}} \circ \cdots \circ  \Phi (\functor{D})_{e_{\target(e),d(e)}}
\end{aligned}
\end{equation*}
(which holds because $S$ is a morphism in $\CKL_{\category{C}}(G_d)$), together with the definition \eqref{eq:in-delay-evn-def} of the maps $\Phi (\functor{D})_{e_{\target(e),n}}$ and $\Phi (\functor{E})_{e_{\target(e),n}}$, ensures that the diagram
\begin{equation}\label{eq:in-delay-full-patch2}
\xymatrix@C=50pt{
\displaystyle \bigsqcup_{\substack{g\in \target^{-1}\target(e)\\ d(g)\geq d(e)}} \functor{D}_{\source(g)} \ar[d]^-{S_{(\target(e),d(e))}} \ar[r]^-{\bigsqcup_g \id_{\functor{D}_{\source(g)}}} 
 &
\displaystyle \bigsqcup_{f\in \target^{-1}\target(e)} \functor{D}_{\source(f)} \ar[d]^-{S_{(\target(e),0)}}
 \\
\displaystyle \bigsqcup_{\substack{g\in \target^{-1}\target(e)\\ d(g)\geq d(e)}} \functor{E}_{\source(g)}\ar[r]^-{\bigsqcup_g \id_{\functor{E}_{\source(g)}} }
 &
\displaystyle \bigsqcup_{f\in \target^{-1}\target(e)} \functor{E}_{\source(f)}
 }
 \end{equation}
 commutes. The definition \eqref{eq:in-delay-comm1} of $T_{\target(e)}$ ensures that the diagram
 \begin{equation}\label{eq:in-delay-full-patch3}
 \xymatrix@C=50pt{
\displaystyle  \bigsqcup_{f\in \target^{-1}\target(e)} \functor{D}_{\source(f)} \ar[d]^-{S_{(\target(e),0)}} \ar[r]^-{\bigsqcup_{f}\functor{D}_f}
  &
  \functor{D}_{\target(e)} \ar[d]^-{T_{\target(e)}} \\
\displaystyle  \bigsqcup_{f\in \target^{-1}\target(e)} \functor{E}_{\source(f)} \ar[r]^-{\bigsqcup_{f}\functor{E}_f}
  &
  \functor{E}_{\target(e)} 
 }
 \end{equation}
 also commutes, and patching together \eqref{eq:in-delay-full-patch1}, \eqref{eq:in-delay-full-patch2}, and \eqref{eq:in-delay-full-patch3} gives a commuting diagram
 \begin{equation*}\label{eq:in-delay-full-patched}
\xymatrix@C=40pt{
\functor{D}_{\source(e)} \ar[d]_-{T_{\source(e)}} \ar[r]^-{\text{canonical}} 
&\displaystyle \bigsqcup_{\substack{g\in \target^{-1}\target(e),\\ d(g)\geq d(e)}} \functor{D}_{\source(g)} \ar[d]^-{S_{(\target(e),d(e))}}  \ar[r]^-{\bigsqcup_g \id} 
&\displaystyle \bigsqcup_{f\in \target^{-1}\target(e)} \functor{D}_{\source(f)} \ar[d]^-{S_{(\target(e),0)}} \ar[r]^-{\bigsqcup_{f} \functor{D}_f} 
& \functor{D}_{\target(e)} \ar[d]^-{T_{\target(e)}} 
\\
\functor{E}_{\source(e)} \ar[r]^-{\text{canonical}} 
&\displaystyle \bigsqcup_{\substack{g\in \target^{-1}\target(e),\\ d(g)\geq d(e)}} \functor{E}_{\source(g)} \ar[r]^-{\bigsqcup_g \id}
&\displaystyle \bigsqcup_{f\in \target^{-1}\target(e)} \functor{E}_{\source(f)} \ar[r]^-{\bigsqcup_{f} \functor{E}_f} 
& \functor{E}_{\target(e)}
}
\end{equation*}
The composition of the horizontal arrows along the top of this last diagram is $\functor{D}_e$, while the composition along the bottom is $\functor{E}_e$, and so we obtain the desired equality $T_{\target(e)}\circ \functor{D}_e = \functor{E}_e \circ T_{\source(e)}$. This shows that our $T$ is indeed a morphism in  $\CKL_{\category{C}}(G)$. 

To complete the proof that our functor $\Phi $ is full, it remains to show that $\Phi (T)=S$. To do this we fix a vertex $v\in G^0$ and a positive integer $n\leq d(v)$. According to \eqref{eq:in-delay-PhiTvn}, and to the diagram \eqref{eq:in-delay-comm1}, the map $\Phi (T)_{(v,n)}$ is the unique morphism 
\[
R : \bigsqcup_{\substack{e\in \target^{-1}(v),\\ d(e)\geq n}} \functor{D}_{\source(e)} \to \bigsqcup_{\substack{e\in \target^{-1}(v),\\ d(e)\geq n}} \functor{E}_{\source(e)}
\] 
with the property that for each $f\in \target^{-1}(v)$ with $d(f)\geq n$, the diagram
\begin{equation}\label{eq:in-delay-full-final}
\xymatrix@C=50pt
{
\displaystyle\bigsqcup_{g\in \target^{-1}\source(f)} \functor{D}_{\source(g)} \ar[d]_-{S_{(\source(f),0)}} \ar[r]^-{\bigsqcup_{g}\functor{D}_g} &\displaystyle \bigsqcup_{\substack{e\in \target^{-1}(v),\\ d(e)\geq n}} \functor{D}_{\source(e)} \ar[d]^-{R} \\
\displaystyle\bigsqcup_{g\in \target^{-1}\source(f)} \functor{E}_{\source(g)}  \ar[r]^-{\bigsqcup_{g}\functor{E}_g} &\displaystyle \bigsqcup_{\substack{e\in \target^{-1}(v),\\ d(e)\geq n}} \functor{E}_{\source(e)} 
}
\end{equation}
commutes. Now the top horizontal arrow in \eqref{eq:in-delay-full-final} is equal to the composition
\[
\Phi (\functor{D})_{e_{v,n+1}} \circ \Phi (\functor{D})_{e_{v,n+2}} \circ \cdots \circ \Phi (\functor{D})_{e_{v,d(f)}} \circ \Phi (\functor{D})_f,
\]
while the bottom horizontal  arrow is equal to the analogous expression with all of the $\functor{D}$s replaced by $\functor{E}$s. The fact that $S$ is a morphism from $\Phi (\functor{D})$ to $\Phi (\functor{E})$  ensures that
\[
\begin{aligned}
& S_{(v,n)}\circ \Phi (\functor{D})_{e_{v,n+1}} \circ \Phi (\functor{D})_{e_{v,n+2}} \circ \cdots \circ \Phi (\functor{D})_{e_{v,d(f)}} \circ \Phi (\functor{D})_f \\
& = \Phi (\functor{E})_{e_{v,n+1}} \circ \Phi (\functor{E})_{e_{v,n+2}} \circ \cdots \circ \Phi (\functor{E})_{e_{v,d(f)}} \circ \Phi (\functor{E})_f \circ S_{(\source(f),0)},
\end{aligned}
\]
and so putting $R=S_{(v,n)}$ into \eqref{eq:in-delay-full-final} does indeed yield a commuting diagram. We conclude that $S_{(v,n)}=\Phi (T)_{(v,n)}$ and that our functor $\Phi $ is full.

The last step in the proof of Theorem \ref{thm:in-delay} is to prove that $\Phi $ is essentially surjective. To do this we fix a diagram $\functor{E}$ in $\CKL_{\category{C}}(G_d)$, and define a diagram $\functor{D}$ of shape $G$ by
\[
\functor{D}_v \coloneqq \functor{E}_{(v,0)}, \qquad \functor{D}_e \coloneqq \begin{cases} \functor{E}_e & \text{if $d(e)=0$,} \\ \functor{E}_{e_{v,1}}\circ \functor{E}_{e_{v,2}}\circ \cdots \circ \functor{E}_{e_{v,d(e)}} \circ \functor{E}_e & \text{if $d(e)\geq 1$}\end{cases}
\]
for each $v\in G^0$ and each $e\in G^1$. The fact that $\functor{E}$ satisfies the coproduct condition ensures that $\functor{D}$ does too, and so $\functor{D}$ is an object of $\CKL_{\category{C}}(G)$. Returning to our graphs \eqref{eq:in-delay-ex1} and \eqref{eq:in-delay-ex2} to illustrate, our construction sends a diagram 
\begin{equation}\label{eq:in-delay-exE}
\functor{E}=\quad \xygraph{
{B_2}="w2":_-{\delta_2}[l(1.5)] {B_1}="w1":_-{\delta_1}[l(1.5)] {B_0}="w0"(:_-{\gamma}@/_15pt/"w2", :^-{\beta}@/^15pt/[l(1.5)]{A}="v0":^-{\alpha}@/^15pt/"w1")
}
\end{equation}
to the diagram 
\begin{equation}\label{eq:in-delay-exD}
\functor{D}=\quad 
\xygraph{
{A}="v":^-{\delta_1\alpha}@/^/[r] {B_0}="w":^-{\beta}@/^/"v","w":^{\delta_1\delta_2\gamma}@(ur,dr)"w"
}.
\end{equation}

We shall construct an isomorphism $T_{(v,n)} : \Phi (\functor{D})_{(v,n)}\to \functor{E}_{(v,n)}$ for each $v\in G^0$ and each $n=0,1,\ldots,d(v)$ recursively, starting at $n=d(v)$. We have
\[
\Phi (\functor{D})_{(v,d(v))} = \bigsqcup_{\substack{e\in \target^{-1}(v)\\ d(e)=d(v)}} \functor{E}_{(\source(e),0)},
\]
and since $\target_d^{-1}(v,d(v))=\{e \ |\ e\in\target^{-1}(v),\ d(e)=d(v)\}$  the coproduct condition on $\functor{E}$ ensures that the map
\[
T_{(v,d(v))}\coloneqq \bigsqcup_{\substack{e\in \target^{-1}(v),\\ d(e)=d(v)}} \functor{E}_e : \bigsqcup_{\substack{e\in \target^{-1}(v)\\ d(e)=d(v)}} \functor{E}_{(\source(e),0)} \to \functor{E}_{(v,d(v))}
\]
is an isomorphism. Now supposing for $n<d(v)$ that we have constructed an isomorphism $T_{(v,n+1)}:\Phi (\functor{D})_{(v,n+1)} \to \functor{E}_{(v,n+1)}$, we construct $T_{(v,n)}$ as follows. We have
\[
\Phi (\functor{D})_{(v,n)} = \bigsqcup_{\substack{e\in \target^{-1}(v)\\ d(e)\geq n}} \functor{E}_{(\source(e),0)} = \left(\bigsqcup_{\substack{e\in \target^{-1}(v)\\ d(e)= n}} \functor{E}_{(\source(e),0)} \right) \sqcup \Phi (\functor{D})_{(v,n+1)},
\]
and we let $T_{(v,n)}: \Phi (\functor{D})_{(v,n)}\to \functor{E}_{(v,n)}$ be the map
\[
T_{(v,n)}\coloneqq \left( \bigsqcup_{\substack{e\in \target^{-1}(v)\\ d(e)= n}} \functor{E}_e \right) \sqcup \left( \functor{E}_{e_{v,n+1}}\circ T_{(v,n+1)}\right).
\]
Since $T_{(v,n+1)}$ is an isomorphism, and since $\target_d^{-1}(v,n)=\{e\in \target^{-1}(v)\ |\ d(e)=n\} \sqcup \{e_{v,n+1}\}$, the coproduct condition on $\functor{E}$ ensures that our $T_{(v,n)}$ is an isomorphism. To illustrate, if $\functor{E}$ and $\functor{D}$ are as in \eqref{eq:in-delay-exE} and \eqref{eq:in-delay-exD}, then  the isomorphisms $T_* : \Phi(\functor{D})_*\to \functor{E}_*$ are the dotted vertical arrows in the diagram
\[
\xygraph{
{B_0}="w2t"(:_-{\text{can}}[l(1.5)] {A\sqcup B_0}="w1t":_-{\id}[l(1.5)] {A\sqcup B_0}="w0t"(:_-{\delta_1\alpha\sqcup \delta_1\delta_2\gamma\,\,}@/_15pt/"w2t", :^-{\delta_1\alpha\sqcup\delta_1\delta_2\gamma}@/^15pt/[l(1.5)]{B_0}="v0t":^-{\beta\quad}@/^20pt/"w1t"),
:^-{\beta}@{.>}[d(2.5)]{B_2}="w2b":_-{\delta_2}[l(1.5)] {B_1}="w1b":_-{\delta_1}[l(1.5)] {B_0}="w0b"(:_-{\gamma}@/_15pt/"w2b", :^-{\beta}@/^15pt/[l(1.5)]{A}="v0b":^-{\alpha\quad}@/^20pt/"w1b")
), "v0t":_-{\beta}@{.>}"v0b", "w0t":^-{\delta_1\alpha\sqcup\delta_1\delta_2\gamma}@{.>}"w0b", "w1t":^-{\alpha\sqcup \delta_2\gamma} @{.>}"w1b"
}
\] 

The last thing to check is that the collection of morphisms 
\[
T=(T_{(v,n)} \ |\ v\in G^0,\ 0\leq n\leq d(v))
\] 
defines a morphism in $\CKL_{\category{C}}(G_d)$. For each edge of the form $e_{v,n}$ in $G_d$ the morphism $T_{(v,n-1)}\circ \Phi (\functor{D})_{e_{v,n}}$ is the composition
\[
\bigsqcup_{\substack{e\in \target^{-1}(v),\\ d(e)\geq n}} \functor{E}_{(\source(e),0)} \xrightarrow{\bigsqcup_e \id} \bigsqcup_{\substack{f\in \target^{-1}(v),\\ d(f)\geq n-1}} \functor{E}_{(\source(f),0)} \xrightarrow{T_{(v,n-1)}}\functor{E}_{(v,n-1)},
\]
and by the recursive definition of $T_{(v,n)}$ this composition is equal to $\functor{E}_{e_{v,n}}\circ T_{(v,n)}$. Thus $T_{\target_d(e_{v,n})}\circ \Phi (\functor{D})_{e_{v,n}} = \functor{E}_{e_{v,n}}\circ T_{\source_d(e_{v,n})}$. On the other hand, tracing through the definitions one finds that for each edge $e\in G^1$ the morphisms 
\[
\xymatrix@C=70pt
{
\displaystyle\bigsqcup_{f\in \target^{-1}\source(e)} \functor{E}_{(\source(f),0)} \ar@<3pt>[r]^-{T_{(\target(e),d(e))}\circ \Phi (\functor{D})_e} \ar@<-3pt>[r]_-{\functor{E}_e \circ T_{(\source(e),0)}} & \functor{E}_{(\target(e),d(e))}
}
\]
are both equal to 
\[
\bigsqcup_{f\in \target^{-1}\source(e)} \functor{E}_e \circ \functor{E}_{(\source(e),1)}\circ \cdots \circ \functor{E}_{(\source(e),d(f))} \circ \functor{E}_f,
\]
showing that $T_{\target_d(e)}\circ \Phi (\functor{D})_e = \functor{E}_e \circ T_{\source_d(e)}$ and completing the verification that $T$ is an isomorphism $\Phi (\functor{D})\to \functor{E}$ in $\CKL_{\category{C}}(G_d)$. Thus $\Phi $ is essentially surjective, and therefore an equivalence.
\end{proof}

\subsection{Out-splitting}

\begin{definition}\label{def:out-split}
Let $G=(G^0,G^1,\source, \target)$ be a directed graph, and let $p:G^0\sqcup G^1\to \N$ be a function such that
\begin{enumerate}[\rm(1)]
\item $p(e)\leq p(\source(e))$ for all $e\in G^1$, and
\item if $v\in G^0$ is  a source then $p(v)=0$.
\end{enumerate}
Define a new directed graph $G_{\mathrm{os},p} = (G_{\mathrm{os},p}^0, G_{\mathrm{os},p}^1, \source_{\mathrm{os},p}, \target_{\mathrm{os},p})$, called the \emph{out-splitting} of $G$ determined by the function $p$, as follows: 
\[
\begin{aligned}
& G_{\mathrm{os},p}^0 \coloneqq \{ (v,n)\ |\ v\in G^0, \ 0\leq n \leq p(v)\}\\
& G_{\mathrm{os},p}^1 \coloneqq \{ (e,n)\ |\ e\in G^1,\ 0 \leq n\leq p(\target(e)) \} \\
& \source_{\mathrm{os},p}(e,n) = (\source(e),p(e)),\quad \target_{\mathrm{os},p}(e,n) = (\target(e), n).
\end{aligned}
\]
\end{definition}

For example, if $G$ is the directed graph
\begin{equation}\label{eq:os-ex1}
\xygraph{
{v}="v":^-{e}@/^/[r] {w}="w":^-{f}@/^/"v","w":^{g}@(ur,dr)"w"
}
\end{equation}
and $p:G^0\sqcup G^1 \to \N$ is the function
\[
p(v)=p(e)=p(g)=0,\qquad p(w)=p(f)=1
\]
then $G_{\mathrm{os},p}$ is the graph
\begin{equation}\label{eq:os-ex2}
\xygraph{
{(v,0)}="v0" ([r(2)][u(.7)]{(w,0)}="w0", [r(2)][d(.7)]{(w,1)}="w1"), "v0":^-{(e,1)}@/^/"w1" :^-{(f,0)}@/^/"v0" :^-{(e,0)}@/^/"w0" :^-{(g,0)}@`{p+(1,1),p+(1,-1)}"w0" :^-{(g,1)}"w1"
}
\end{equation}

\begin{theorem}\label{thm:out-split}
Let $G$ be a directed graph, and let $G_{\mathrm{os},p}$ be an out-splitting of $G$ as in Definition \ref{def:in-delay}. For each category $\category{C}$ there is an equivalence 
\(
\CKL_{\category{C}}(G) {\cong} \CKL_{\category{C}}(G_{\mathrm{os},p}).
\)
\end{theorem}

\begin{proof}
To simplify the notation we will write $G_p$ instead of $G_{\mathrm{os},p}$.

For each diagram $\functor{D}$  in $\CKL_{\category{C}}(G)$ we define a diagram $\Phi (\functor{D})$ in $\CKL_{\category{C}}(G_p)$ by
\[
\Phi (\functor{D})_{(v,n)}\coloneqq \functor{D}_v \quad \text{and}\quad \Phi (\functor{D})_{(e,n)} \coloneqq \functor{D}_e.
\]
Since each vertex $(v,n)$ of $G_p$ has $\target_p^{-1}(v,n) = \{ (e,n)\ |\ e\in \target^{-1}(v)\}$, the diagram $\CKL_{\category{C}}(G)$ is easily seen to inherit the coproduct property from $\functor{D}$.
For the graphs $G$ and $G_p$ in \eqref{eq:os-ex1} and \eqref{eq:os-ex2}, for instance, $\Phi$ sends each diagram
\[
\xygraph{
{X}="v":^-{\phi}@/^/[r] {Y}="w":^-{\psi}@/^/"v","w":^{\rho}@(ur,dr)"w"
}
\]
to the diagram
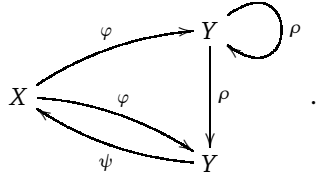
\[
\xygraph{
{X}="v0" ([r(2)][u(.7)]{Y}="w0", [r(2)][d(.7)]{Y}="w1"), "v0":^-{\phi}@/^/"w1" :^-{\psi}@/^/"v0" :^-{\phi}@/^/"w0" :^-{\rho}@`{p+(1,1),p+(1,-1)}"w0" :^-{\rho}"w1"
}.
\]

To make $\Phi $ into a functor we define, for each morphism $T:\functor{D}\to \functor{E}$ in $\CKL_{\category{C}}(G)$, a morphism $\Phi (T) : \Phi (\functor{D})\to \Phi (\functor{E})$ by
\[
\Phi (T)_{(v,n)} \coloneqq T_v : \functor{D}_v \to \functor{E}_v.
\]
The fact that $\Phi (T)$ is a morphism of diagrams follows easily from the corresponding property of $T$.  
Since each component $T_v$ of the morphism $T$ appears as a component $\Phi(T)_{(v,0)}$ of the morphism $\Phi (T)$, the functor $\Phi $ is faithful.

To see that $\Phi $ is full, let $\functor{D}$ and $\functor{E}$ be objects in $\CKL_{\category{C}}(G)$, and let $S: \Phi (\functor{D})\to \Phi (\functor{E})$ be a morphism in $\CKL_{\category{C}}(G_p)$. For each vertex $v\in G^0$ that is not a source, each $n\in \{0,1,\ldots,p(n)\}$, and each $e\in \target^{-1}(v)$ the equality $S_{\target_p(e,n)} \circ \Phi (\functor{D})_{(e,n)} = \Phi (\functor{E})_{(e,n)} \circ S_{\source_p(e,n)}$ implies that $S_{(v,n)}\circ \functor{D}_e = \functor{E}_e \circ S_{(\source(e),p(e))}$. Taking a coproduct over $e\in \target^{-1}(v)$ shows that the diagram
\[
\xymatrix@C=50pt
{
\displaystyle\bigsqcup_{e\in \target^{-1}(v)} \functor{D}_{\source(e)} \ar[r]^-{\bigsqcup_{e} \functor{D}_e} \ar[d]_-{\bigsqcup_{e} S_{(\source(e),p(e))}} 
& \functor{D}_v \ar[d]^-{S_{(v,n)}} \\
\displaystyle \bigsqcup_{e\in \target^{-1}(v)} \functor{E}_{\source(e)} \ar[r]^-{\bigsqcup_{e} \functor{E}_e}
&
\functor{E}_v
}
\]
commutes. Since the horizontal arrows are isomorphisms and everything but $S_{(v,n)}$ is independent of $n$, we conclude that 
\begin{equation}\label{eq:in-split-indep-n}
S_{(v,n)}=S_{(v,m)}\quad \text{for all $n,m\in \{0,1,\ldots,p(v)\}$.}
\end{equation}
On the other hand, if $v\in G^0$ is a source then we required that $p(v)=0$ and so the equality \eqref{eq:in-split-indep-n} is clearly satisfied. With this equality in hand it is easily checked that setting $T_v\coloneqq S_{(v,0)}$ for each $v\in G^0$ yields a morphism $T:\functor{D}\to \functor{E}$ in $\CKL_{\category{C}}(G)$ such that $\Phi (T)=S$, and so $\Phi $ is full.

Finally, to show that $\Phi $ is essentially surjective, let $\functor{E}$ be an object of $\CKL_{\category{C}}(G_p)$. For each vertex $v\in G^0$ that is not a source, and each  $n\in \{0,1,\ldots,p(v)\}$, we let $T_{(v,n)} : \functor{E}_{(v,0)} \to \functor{E}_{(v,n)}$ be the isomorphism in $\category{C}$ making the diagram
\begin{equation}\label{eq:out-split-psi}
\xymatrix@C=60pt{
\displaystyle \bigsqcup_{e\in \target^{-1}(v)} \functor{E}_{(\source(e),p(e))} \ar[r]_{\cong}^-{\bigsqcup_e \functor{E}_{(e,0)}} \ar[dr]_-{\bigsqcup_e \functor{E}_{(e,n)}}^-{\cong} & \functor{E}_{(v,0)} \ar[d]^-{T_{(v,n)}} \\
& \functor{E}_{(v,n)}
}
\end{equation}
commute. For each source $v\in G^0$ we let $T_{(v,0)}:\functor{E}_{(v,0)}\to \functor{E}_{(v,0)}$ be the identity map. We then let $\functor{D}$ be the diagram of shape $G$ given by
\[
\functor{D}_v = \functor{E}_{(v,0)} \quad \text{and}\quad \functor{D}_e = \functor{E}_{(e,0)}\circ T_{(\source(e),p(e))}.
\]
The diagram $\functor{D}$ satisfies the coproduct condition, since at each vertex $v\in G^0$ that is not a source the morphism $\bigsqcup_{e\in \target^{-1}(v)} \functor{D}_e : \bigsqcup_{e\in \target^{-1}(v)}\functor{D}_{\source(e)} \to \functor{D}_v$ is the composition
\[
\bigsqcup_{e\in \target^{-1}(v)} \functor{E}_{(\source(e),0)} \xrightarrow{\bigsqcup_e T_{(\source(e),p(e))}} \bigsqcup_{e\in \target^{-1}(v)} \functor{E}_{(\source(e),p(e))} \xrightarrow{\bigsqcup_e \functor{E}_{(e,0)}} \functor{E}_{(v,0)}
\]
where the first arrow is an isomorphism by definition of the $T$s, and the second arrow is an isomorphism because $\functor{E}$ satisfies the coproduct condition. 

Finally, we will show that the isomorphisms
\[
\Phi (\functor{D})_{(v,n)} = \functor{E}_{(v,0)} \xrightarrow{T_{(v,n)}} \functor{E}_{(v,n)},
\]
defined for each vertex $(v,n)\in G_p^0$, assemble into a  morphism $T:\Phi (\functor{D})\to \functor{E}$ in $\CKL_{\category{C}}(G_p)$, which will show that our functor $\Phi $ is essentially surjective and thus an equivalence. We must check that for each $e\in G^1$ and each $n\in \{0,1,\ldots,p(\target(e))\}$ we have $\functor{E}_{(e,n)}\circ T_{(\source(e),p(e))} = T_{(\target(e),n)} \circ \Phi (\functor{D})_{(e,n)}$. Since $\Phi (\functor{D})_{(e,n)} = \functor{E}_{(e,0)}\circ T_{(\source(e),p(e))}$ and $T_{(\source(e),p(e))}$ is an isomorphism, the question becomes whether $\functor{E}_{(e,n)} = T_{(\target(e),n)} \circ \functor{E}_{(e,0)}$. This last equality is an immediate consequence of the definition \eqref{eq:out-split-psi} of $T_{(\target(e),n)}$, and so $T$ is indeed a morphism in $\CKL_{\category{C}}(G_p)$.
\end{proof}

\subsection{In-splitting}\label{subsec:in-split}

\begin{definition}\label{def:in-split}
Let $G=(G^0,G^1,\source, \target)$ be a directed graph, and let $p:G^0\sqcup G^1\to \N$ be a function such that
\begin{enumerate}[\rm(1)]
\item $p(e)\leq p(\target(e))$ for all $e\in G^1$,
\item if $v\in G^0$ is a source then $p(v)=0$, and
\item if $v\in G^0$ is not a source then the function $p:\target^{-1}(v)\to \{0,1,\ldots,p(v)\}$ is surjective.
\end{enumerate}
Define a new graph $G_{\mathrm{is},p}=(G^0_{\mathrm{is},p}, G^1_{\mathrm{is},p}, \source_{\mathrm{is},p},\target_{\mathrm{is},p})$, called the \emph{in-splitting} of $G$ determined by the function $p$, as follows: 
\[
\begin{aligned}
& G^0_{\mathrm{is},p} \coloneqq \{ (v,n)\ |\ v\in G^0,\ 0\leq n\leq p(v) \} \\
& G^1_{\mathrm{is},p} \coloneqq \{ (e,n)\ |\ e\in G^1,\ 0\leq n\leq p(\source(e)) \}\\
& \source_{\mathrm{is},p}(e,n)=(\source(e),n),\quad \target_{\mathrm{is},p}(e,n) = (\target(e),p(e)).
\end{aligned}
\]
\end{definition}

For example, if $G$ is the directed graph
\begin{equation}\label{eq:is-ex1}
\xygraph{
{v}="v":^-{e}@/^/[r] {w}="w":^-{f}@/^/"v","w":^{g}@(ur,dr)"w"
}
\end{equation}
and $p:G^0\sqcup G^1 \to \N$ is the function
\[
p(v)=p(f)=p(g)=0,\qquad p(w)=p(e)=1
\]
then $G_{\mathrm{is},p}$ is the graph
\begin{equation}\label{eq:is-ex2}
\xygraph{
{(v,0)}="v0" ([r(2)][u(.7)]{(w,0)}="w0", [r(2)][d(.7)]{(w,1)}="w1"), "w1":_-{(g,1)}"w0" :^-{(g,0)}@`{p+(1,1),p+(1,-1)}"w0":_-{(f,0)}@/_/"v0":^-{(e,0)}@/^/"w1" :^-{(f,1)}@/^/"v0"
}
\end{equation}

\begin{theorem}\label{thm:in-split}
Let $G_{\mathrm{is},p}$ be an in-splitting of a directed graph $G$ as in Definition \ref{def:in-split}, where $G$ has no infinite recievers. For each category $\category{C}$ with binary coproducts there is an equivalence 
\(
 \CKL_{\category{C}}(G) {\cong} \CKL_{\category{C}}(G_{\mathrm{is},p}).
\)
\end{theorem}

\begin{proof}
We write $G_p$ instead of $G_{\mathrm{is},p}$. As in the proof of Theorem \ref{thm:in-delay} we may assume without loss of generality that $G$ has no sources: add a head at each source using Theorem \ref{thm:heads}, and extend the function $p$ to the new graph by setting $p=0$ on the added vertices and edges. 

For each diagram $\functor{D}$ in $\CKL_{\category{C}}(G)$ we define a diagram $\Phi (\functor{D})$ of shape $G_p$ by setting
\[
\Phi (\functor{D})_{(v,n)}\coloneqq \bigsqcup_{\substack{e\in\target^{-1}(v),\\ p(e)=n}} \functor{D}_{\source(e)}
\]
for each $(v,n)\in G^0_p$, and
\begin{equation}\label{eq:in-split-ess-eq-1} 
\Phi (\functor{D})_{(e,n)} \coloneqq \bigsqcup_{\substack{f\in \target^{-1}\source(e),\\ p(f)=n}} \functor{D}_f : \bigsqcup_{\substack{f\in \target^{-1}\source(e),\\ p(f)=n}} \functor{D}_{\source(f)} \to \bigsqcup_{\substack{g\in \target^{-1}\source(e),\\ p(g)=p(e)}} \functor{D}_{\source(g)}
\end{equation}
for each $(e,n)\in G_p^1$. For the graphs $G$ and $G_p$ in \eqref{eq:is-ex1} and \eqref{eq:is-ex2}, for instance, $\Phi$ sends each diagram
\[
\xygraph{
{X}="v":^-{\phi}@/^/[r] {Y}="w":^-{\psi}@/^/"v","w":^{\rho}@(ur,dr)"w"
}
\]
to the diagram
\[
\xygraph{
{Y}="v0" ([r(2)][u(.7)]{Y}="w0", [r(2)][d(.7)]{X}="w1"), "w1":_-{\phi}"w0" :^-{\rho}@`{p+(1,1),p+(1,-1)}"w0":_-{\rho}@/_/"v0":^-{\psi}@/^/"w1" :^-{\phi}@/^/"v0"
}.
\]
Note that the surjectivity condition (3) in Definition \ref{def:in-split}, along with our assumption that $G$ has no sources or infinite receivers, ensures that all of the coproducts appearing in the definition of $\Phi$ are over nonempty finite sets. Note too that $\Phi (\functor{D})_{(e,n)}$ is well-defined as a morphism in $\category{C}$ from $\Phi (\functor{D})_{(\source(e),n)}$ to $\Phi (\functor{D})_{(\target(e),p(e))}$, since each $\functor{D}_f$ maps $\functor{D}_{\source(f)}$ to $\functor{D}_{\source(e)}$, and $e$ is one of the indices $g$ appearing in the right-hand coproduct. The diagram $\Phi (\functor{D})$ is easily seen to satisfy the coproduct condition, and thus is an object of $\CKL_{\category{C}}(G_p)$.

For each morphism $T:\functor{D}\to \functor{E}$ in $\CKL_{\category{C}}(G)$ we let $\Phi (T)$ be the morphism in $\CKL_{\category{C}}(G_p)$ from $\Phi (\functor{D})$ to $\Phi (\functor{E})$ defined at each vertex $(v,n)\in G_p^0$ by
\[
\Phi (T)_{(v,n)} \coloneqq \bigsqcup_{\substack{e\in \target^{-1}(v),\\ p(e)=n}} T_{\source(e)} : \bigsqcup_{\substack{e\in \target^{-1}(v),\\ p(e)=n}} \functor{D}_{\source(e)} \to \bigsqcup_{\substack{e\in \target^{-1}(v),\\ p(e)=n}} \functor{E}_{\source(e)}.
\]
In this way we obtain a functor $\Phi :\CKL_{\category{C}}(G)\to \CKL_{\category{C}}(G_p)$.

An argument like the one in the proof of Theorem \ref{thm:in-delay} shows that $\Phi $ is faithful. To see that this functor is full, let $\functor{D}$ and $\functor{E}$ be objects in $\CKL_{\category{C}}(G)$, and let $S:\Phi (\functor{D})\to \Phi (\functor{E})$ be a morphism in $\CKL_{\category{C}}(G_p)$. Noting that for each $v\in G^0$ we have
\[
\bigsqcup_{0\leq n\leq p(v)} \Phi (\functor{D})_{(v,n)} = \bigsqcup_{e\in \target^{-1}(v)} \functor{D}_{\source(e)}
\]
(along with an analogous equality for $\functor{E}$), we let $T_v : \functor{D}_v\to \functor{E}_v$ be the morphism in $\category{C}$ making the diagram
\[
\xymatrix@C=50pt
{
\displaystyle\bigsqcup_{e\in \target^{-1}(v)} \functor{D}_{\source(e)} \ar[d]_-{\bigsqcup_e S_{(v,n)}} \ar[r]^-{\bigsqcup_e \functor{D}_e}_-{\cong} & \functor{D}_v \ar[d]^-{T_v} \\
\displaystyle \bigsqcup_{e\in \target^{-1}(v)} \functor{E}_{\source(e)}  \ar[r]^-{\bigsqcup_e \functor{E}_e}_-{\cong} & \functor{E}_v
}
\]
commute. It is a straightforward matter to check that the collection of morphisms $(T_v\ |\ v\in G^0)$ defines a morphism $T:\functor{D}\to \functor{E}$ in $\CKL_{\category{C}}(G)$, such that $\Phi (T)=S$. Thus $\Phi $ is full.

Finally, to show that $\Phi $ is essentially surjective, we let $\functor{E}$ be an object of $\CKL_{\category{C}}(G_p)$, and for each vertex $v\in G^0$ we define
\[
\functor{D}_v \coloneqq \bigsqcup_{0\leq n\leq p(v)} \functor{E}_{(v,n)},
\]
while for each edge $e\in G^1$ we define
\begin{equation}\label{eq:in-split-ess-eq-2}
\functor{D}_e \coloneqq \bigsqcup_{0\leq n\leq p(\source(e))} \functor{E}_{(e,n)} : \bigsqcup_{0\leq n\leq p(\source(e))} \functor{E}_{(\source(e),n)} \to \bigsqcup_{0\leq m\leq p(\target(e))} \functor{E}_{(\target(e),m)}.
\end{equation}
(Note that $\functor{E}_{(e,n)}$ maps $\functor{E}_{(\source(e),n)}$ to $\functor{E}_{(\target(e),p(e))}$, where $p(e)$ is one of the indices $m$ in the coproduct on the right-hand side.) It is easily checked that this diagram $\functor{D}$ satisfies the coproduct condition. To see that $\Phi (\functor{D})\cong \functor{E}$, we note that for each vertex $(v,n)\in G_p^0$ we have
\[
\Phi (\functor{D})_{(v,n)} = \bigsqcup_{\substack{e\in \target^{-1}(v),\\ p(e)=n, \\ 0\leq m\leq p(\source(e))}} \functor{E}_{(\source(e),m)}.
\]
The coproduct on the right-hand side is over the set 
\[
\target_p^{-1}(v,n)=\{(e,m)\ |\ e\in \target^{-1}(v),\ p(e)=n,\ 0\leq m\leq p(\source(e))\},
\] 
and so the coproduct condition on $\functor{E}$ ensures that the map
\[
T_{(v,n)}\coloneqq \bigsqcup_{\substack{e\in \target^{-1}(v),\\ p(e)=n, \\ 0\leq m\leq p(\source(e))}} \functor{E}_{(e,m)} : \Phi (\functor{D})_{(v,n)} \to \functor{E}_{(v,n)}
\] 
is an isomorphism. A straightforward verification using the definitions \eqref{eq:in-split-ess-eq-1} and \eqref{eq:in-split-ess-eq-2} shows that the collection of morphisms $(T_{(v,n)}\ |\ (v,n)\in G_p^0)$ defines an isomorphism $T:\Phi (\functor{D})\to \functor{E}$, and so $\Phi $ is essentially surjective.
\end{proof}

\subsection{Desingularisation}\label{subsec:desingularisation} 

Here is an example of a graphical construction that is known to yield Morita equivalences of Leavitt path algebras, but which does not yield equivalences of the categories $\CKL_{\category{C}}$ in general, even when $\category{C}$ is assumed to have all coproducts.

Consider these two graphs:
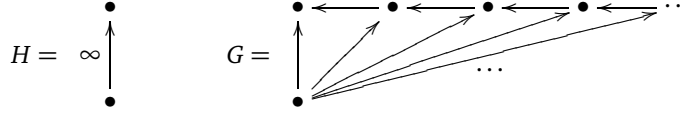
\[
H=\
\xygraph{
[u(.5)]{\closed}="dot" [d]{\closed}="v":^-{\infty}"dot"
}\qquad\qquad
G=\ 
\xygraph{
[u(.5)]{\closed}="y"([d]{\closed}="x",[r]{\closed}="b1" [r]{\closed}="b2" [r]{\closed}="b3" [r]{\cdots}="dots"), "x":"y", "x":"b1":"y", "x":"b2":"b1", "x":"b3":"b2", "x":_-*{\cdots}"dots":"b3"
}
\]
where the label $\infty$ in $H$  indicates that there are countably infinitely many edges from the lower vertex to the upper vertex. Note that $G$ is the graph from Remark \ref{remark:P-infinite}.

The graphs $H^{\opp}$ and $G^{\opp}$ are related by the \emph{desingularisation} construction of \cite{Drinen-Tomforde}, which was shown in \cite[Theorem]{Abrams-desingularisation} to yield Morita equivalences of Leavitt path algebras. Combining this fact with Theorem \ref{thm:LPA-equivalence} we obtain a chain of equivalences
\[
\CKL_{\Mod(\mathbb{F})}(G) \cong \Mod(L_{\mathbb{F}}(G^{\opp})) \cong \Mod(L_{\mathbb{F}}(H^{\opp})) \cong 
\CKL_{\Mod(\mathbb{F})}(H_+)
\]
for each field $\mathbb{F}$. 

On the other hand, if $\category{P}$ is the category associated to a partially ordered set $(P,\leq)$ having all finite suprema, then Corollary \ref{cor:P-finite} gives $\CKL_{\category{P}}(H_+) \cong \category{P}^2$, while we observed in Remark \ref{remark:P-infinite} that $\CKL_{\category{P}}(G)\cong \category{Arr}(\category{P})$, and that there are many examples of partially ordered sets (having all suprema) for which $\category{Arr}(\category{P}) \not\cong \category{P}^2$. For such $\category{P}$ we will have $\CKL_{\category{P}}(G)\not\cong \CKL_{\category{P}}(H_+)$.

\subsection{The invariants of Parry-Sullivan and Bowen-Franks}\label{subsec:PSBF}

Let $G$ be a finite directed graph. The \emph{adjacency matrix} $A_G$ of $G$ is the $\#G^0\times \#G^0$ matrix whose $(v,w)$ entry is equal to the number of edges in $G$ with source $v$ and target $w$. Writing down such a matrix involves choosing an ordering of $G^0$, but nothing essential will depend on that choice. The graph $G$ is called \emph{irreducible} if for each $(v,w)\in G^0\times G^0$ there is a directed path in $G$ with source $v$ and target $w$ (see the discussion preceding Corollary \ref{cor:P-finite}); and if $G$ is irreducible then it is called \emph{non-trivial} if $A_G$ is not a permutation matrix.

\begin{definition}
If $G$ is a finite directed graph with $n$ vertices then the \emph{Parry-Sullivan number} $\PS(G)$ and the \emph{Bowen-Franks group} $\BF(G)$ are defined by
\[
\PS(G)\coloneqq \det(I_n-A_G) \qquad \text{and}\qquad \BF(G) \coloneqq \Z^n / (I_n-A_G)\Z^n
\]
where $I_n$ is the $n\times n$ identity matrix.
\end{definition}

Franks proved in \cite{Franks} that if two irreducible, non-trivial directed graphs have equal Parry-Sullivan numbers and isomorphic Bowen-Franks groups, then each of those graphs can be transformed into the other by a sequence of out-/in-delays, out-/in-splittings, and the inverses of these moves. In \cite{Abrams-FE} it was shown that these graph moves preserve Morita equivalence of Leavitt path algebras, and consequently that if the Parry-Sullivan and Bowen-Franks invariants of two irreducible, non-trivial graphs coincide, then the Leavitt path algebras of those graphs (over an arbitrary field) are Morita equivalent. In Theorems \ref{thm:out-delay}, \ref{thm:in-delay}, \ref{thm:out-split}, and \ref{thm:in-split}  we have shown that for finite graphs the moves in question in fact give equivalences of $\CKL_{\category{C}}$ categories whenever $\category{C}$ has binary coproducts, and so we obtain:

\begin{corollary}\label{cor:PSBF}
Let $G$ and $H$ be irreducible, non-trivial finite directed graphs. If $\PS(G)=\PS(H)$ and $\BF(G)\cong \BF(H)$ then $\CKL_{\category{C}}(G) \cong \CKL_{\category{C}}(H)$ for all categories $\category{C}$ with binary coproducts.\hfill\qed
\end{corollary}

We do not know whether the converse of Corollary \ref{cor:PSBF} is true. It is known that if $\CKL_{\category{\Mod(\mathbb{F})}}(G) \cong \CKL_{\category{\Mod(\mathbb{F})}}(H)$ then $\BF(G)\cong \BF(H)$, because in this case the Leavitt path algebras $L_{\mathbb{F}}(G^{\opp})$ and $L_{\mathbb{F}}(H^{\opp})$ are Morita equivalent (Theorem \ref{thm:LPA-equivalence}), and the Bowen-Franks groups are the $K_0$-groups of these algebras (see for example \cite{Abrams-classification}). So a positive answer to Question \ref{q:PS} from the introduction (is there a category $\category{C}$ with binary coproducts such that the Parry-Sullivan number $\PS(G)$ is an invariant of the category $\CKL_{\category{C}}(G)$?) would imply the converse of Corollary \ref{cor:PSBF} (if $\CKL_{\category{C}}(G)\cong \CKL_{\category{C}}(H)$ for all $\category{C}$ then $\PS(G)=\PS(H)$ and $\BF(G)\cong \BF(G)$). As we explained in the introduction, we hope that an answer to this question will help to clear up some of the mystery surrounding the role of the Parry-Sullivan number in the classification of Leavitt path algebras.

\bibliographystyle{alpha}
\bibliography{Leavitt.bib}

 \end{document}